\documentclass[11pt]{article}
\usepackage{amsfonts,amsthm}
\usepackage{amssymb}
\usepackage{amsmath}
\usepackage{dutchcal}
\usepackage{mathrsfs}

\usepackage{hyperref}
\hypersetup{colorlinks=true,linkcolor=red,citecolor=red,filecolor=magenta,urlcolor=blue}
\usepackage{xcolor}
\usepackage{cleveref}
\usepackage{comment}
\newcommand{\mc}[1]{\mathcal{#1}}

\def\sqr#1#2{{\vcenter{\vbox{\hrule height.#2pt
              \hbox{\vrule width.#2pt height#1pt \kern#1pt \vrule
width.#2pt}
              \hrule height.#2pt}}}}
\def\dbN{{\mathbb{N}}}

\def\3n{\negthinspace \negthinspace \negthinspace }
\def\2n{\negthinspace \negthinspace }
\def\1n{\negthinspace }
\def\ns{\noalign{\smallskip} }
\def\ns{\noalign{\medskip} }

\def\no{\noindent}

\def\ss{\smallskip}

\def\q{\quad}

\def\liminf{\mathop{\underline{\rm lim}}}

\def\bel{\begin{equation}\label}
\def\ee{\end{equation}}
\def\bea{\begin{eqnarray}}
\def\eea{\end{eqnarray}}
\def\bt{\begin{theorem}}
\def\et{\end{theorem}}
\def\bc{\begin{corollary}}
\def\ec{\end{corollary}}
\def\bl{\begin{lemma}}
\def\el{\end{lemma}}
\def\bp{\begin{proposition}}
\def\ep{\end{proposition}}
\def\br{\begin{remark}}
\def\er{\end{remark}}
\def\ba{\begin{array}}
\def\ea{\end{array}}
\def\bd{\begin{definition}}
\def\ed{\end{definition}}

\newtheorem{lemma}{Lemma}[section]
\newtheorem{remark}{Remark}[section]

\newtheorem{theorem}{Theorem}[section]
\newtheorem{corollary}{Corollary}[section]

\newtheorem{definition}{Definition}[section]
\newtheorem{proposition}{Proposition}[section]

\makeatletter
   
   \@addtoreset{equation}{section}
\makeatother

\usepackage{enumitem}
\begin{document}

\title{\bf   Quantitative linear approximation  for controllability of quasi-linear parabolic equations\thanks{This work was partially supported by the National Key R\&D Program of China
under grants 2023YFA1009002 and 2024YFA1013400,  the NSF of China under grant 12371444, and the
New Cornerstone Investigator Program.}}

\author{Manish Kumar\thanks{School of Mathematics,  Sichuan University,
Chengdu  610064, China. E-mail address: gomanishnu@gmail.com.}, \quad  Xu Liu\thanks{School of Mathematics and Statistics, Northeast Normal
University, Changchun 130024, China.   E-mail address:
liux216@nenu.edu.cn.}   \quad and \quad  Xu Zhang\thanks{School of Mathematics and New Cornerstone Science Laboratory,  Sichuan University,
Chengdu  610064, China. E-mail address: zhang\_xu@scu.edu.cn.}}

\date{}

\maketitle

\begin{abstract}
\no
This paper establishes
 a quantitative  relationship between the null
  controllability of a quasi-linear parabolic equation and  its linear counterpart.
  Under suitable assumptions and for sufficiently small initial data,
  we prove that  both
  systems are null controllable, and that the differences between
  their corresponding controls and states satisfy a quadratic error estimate.
  Unlike standard controllability problems,
  the controls constructed here are designed  not only to steer the
  systems to a prescribed target at the
  given terminal time, but  also  to ensure that the deviation between the nonlinear and linear dynamics obeys this estimate.

To prove this approximation result,
we develop two complementary methods.
The first  is an indirect,  modular  time-splitting approach that
  reveals  a transfer mechanism: a control for the linear system is initially
   applied to both systems,
    a zero-control interval propagates the resulting quadratic state discrepancy,
    and a final local control eliminates the remaining nonlinear state.
The second  is a direct approach utilizing
  Carleman estimates and fixed-point techniques.
This method  establishes a uniform comparison for the relevant linearized systems
and  yields    new
Carleman estimates for parabolic operators with $C^{1, 1}$-principal coefficients.

Although both approaches  lead to the same quadratic approximation,
they possess  critical  differences.
The indirect method can be generalized to
other nonlinear evolution equations, provided
 suitable local controllability and same-control stability
results are available. Conversely,
 the direct approach supplies  detailed analytic estimates and explicitly
demonstrates  the dependence of the error estimate on the principal coefficients.
 \end{abstract}

\medskip

\no{\bf Key Words}.
Quasi-linear parabolic equation,  linear parabolic equation,
     null controllability,  Carleman estimate

\medskip

\no{\bf AMS subject classifications}. 93B05,  93C20,  35K59

\date{}
\maketitle

\section{Introduction}

Let  $T>0$, $n\in\mathbb N$,  and  $\Omega\subseteq\mathbb{R}^n$ be a   bounded domain
  with a $C^4$ boundary $\Gamma$.
Set $Q=\Omega\times(0,T)$ and $\Sigma=\Gamma\times(0,T)$.
Let $\omega_0$ and
  $\omega$ be two  non-empty open subsets of $\Omega$ such that
  $\overline{\omega_0}\subseteq\omega$. Let
$\xi_{0}\in C_0^\infty(\omega)$ be a given cutoff  function satisfying   $\xi_{0}=1$
 in $\omega_0$ and $0\leq \xi_0\leq 1$ in $\omega$.
We consider the following controlled quasi-linear parabolic equation:
\begin{eqnarray}\label{aeq1}
	\left\{
	\begin{array}{ll}
		y_t-\sum\limits_{i, j=1}^n \left(a^{ij}(y, \nabla y)y_{x_i}\right)_{x_j} + f(y, \nabla y)=\xi_0 u &\mbox{ in }Q, \\ \ns
		y=0 &\mbox{ on }\Sigma, \\ \ns
		y(x, 0)=y_0(x) &\mbox{ in }\Omega,
	\end{array}
	\right.
\end{eqnarray}
where $u$ is the control  variable,  $y$ is the state variable,
$y_0$ is the initial value, $f\in C^3(\mathbb R^{1+n})$ with $f(0, {\bf 0})=0$,
and $a^{i j}\in C^3(\mathbb R^{1+n})$  with $a^{i j}=a^{ji}$ for all $i, j=1, \cdots, n$. Moreover,
there exists a positive constant $\rho_0$ such that
\begin{align}\label{a_ij_cond}
	\sum\limits_{i, j=1}^n a^{i j}(s, \zeta)\eta_i\eta_j
	\geq \rho_0|\eta|^2, \quad\forall\ (s, \zeta, \eta)=(s, \zeta_1, \cdots, \zeta_n,  \eta_1, \cdots, \eta_n)\in
	\mathbb{R}\times\mathbb{R}^n \times\mathbb{R}^n.
\end{align}
In what follows, for a $C^1$ function $g=g(s, \zeta)=g(s, \zeta_1, \cdots, \zeta_n)$ defined
in $\mathbb{R}^{1+n}$,
 we  denote by $g_s(\cdot, \cdot)$ and $g_{\zeta_k}(\cdot, \cdot)$ its
  partial derivatives
 with respect to the first variable $s\in \mathbb R$ and
  the $(k+1)$-th variable $\zeta_k\in \mathbb R$
   $(k=1, \cdots, n)$, respectively.
  Let $\nabla g$ and $\nabla_\zeta g$ denote
  the gradient of $g$ with respect to all variables and  its gradient with
  respect to the last $n$
  variables, respectively.
 Furthermore, for a $C^2$ function $\tilde g=\tilde g(x, t)$ defined in $Q$, we use
 $\nabla \tilde g$ and $\Delta \tilde g$  to denote its gradient and Laplacian with respect to the spatial variable $x\in \Omega$.

Quasi-linear parabolic equations
  describe a wide range of natural diffusion  phenomena.
For instance,
 let $y$, $c$ and $\rho$ denote
 the temperature distribution,  specific heat,  and
  density
   of a
  body $\Omega$, respectively.
  By  energy
   conservation and   Fourier's  law,
   we obtain
$
 c \rho y_t-\sum\limits_{i, j=1}^n (a^{ij}y_{x_i})_{x_j}+f=\xi_0 u,
$
where $A\triangleq\left(a^{i j}\right)_{1\leq i, j\leq n}$ is the  thermal conductivity
 matrix,    while
 $f$ and $u$ represent  heat sources, with
 $u$  being artificially regulated  inside the localized domain $\omega$.
When
  the thermal conductivity matrix $A$   and the heat source $f$
   depend  on
 the temperature distribution $y$ or  its  gradient $\nabla y$,
   the  thermal  diffusion process
  is governed by
\begin{equation}\label{DD1}
c \rho y_t-\sum_{i, j=1}^n \left(a^{ij}(y, \nabla y)y_{x_i}\right)_{x_j}+f(y, \nabla y)=\xi_0 u.
\end{equation}
When $A$ is symmetric and
uniformly positive definite, (\ref{DD1}) becomes  a controlled quasi-linear parabolic equation.
For simplicity,  we assume that
   $c\rho=1$, in which case  (\ref{DD1}) reduces to  the first equation in (\ref{aeq1}).

Note that when both $y$ and $|\nabla y|$ are sufficiently small,
  $a^{i j}(y, \nabla y)$ and $f(y, \nabla y)$ can be formally approximated by
   $\tilde{a}^{i j}\triangleq a^{i j}(0, {\bf 0})$ and $f(0, {\bf 0})=0$, respectively.
Consequently,
  (\ref{aeq1})
can be formally  approximated  by the following controlled linear parabolic equation:
   \begin{eqnarray}\label{aeq1*}
	\left\{
	\begin{array}{ll}
		w_t-\sum\limits_{i, j=1}^n \left(\tilde{a}^{ij}w_{x_i}\right)_{x_j}=\xi_0 v &\mbox{ in }Q, \\ \ns
		w=0 &\mbox{ on }\Sigma, \\ \ns
		w(x, 0)=y_0(x) &\mbox{ in }\Omega.
	\end{array}
	\right.
\end{eqnarray}
Here, we use $w$ and $v$ in (\ref{aeq1*}) to replace  $y$ and $u$ in (\ref{aeq1})
 to avoid confusing notation.
Note that the initial data $y_0(\cdot)$ in both (\ref{aeq1}) and (\ref{aeq1*})
are chosen to be identical.

The purpose of this paper
 is to provide  a rigorous proof of the
 formal  model approximation derived above and,  more importantly,
  to establish a precise connection between the null controllability of the quasi-linear parabolic equation \eqref{aeq1} and that of the linear heat equation \eqref{aeq1*}.
  More precisely, for sufficiently small initial data in a suitable H\"older space,
  we construct controls $u$ and $v$ such that the associated systems
  \eqref{aeq1}) and \eqref{aeq1*} both reach the zero state at time $T$.
Moreover, we show that the gap between the corresponding solutions is  of
quadratic order  with respect to the size of  initial data.
Thus, in the small-data regime,
 the null controllability of the quasi-linear parabolic system \eqref{aeq1} can be accurately
  approximated by that of the linear parabolic system \eqref{aeq1*} with negligible error.
While controllability problems for quasi-linear parabolic equations
have been extensively studied in the literature (see \cite{MB}, \cite{DLZ}, \cite{f}, \cite{Liu}, \cite{Liux}, \cite{N}, \cite{Ni}, \cite{z} and the references therein), to the best of our knowledge, the relationship between the controllability of quasi-linear and linear systems has not been previously investigated.

Next, we introduce some standard notation.  For any $k,
\ell\in \dbN$,  we denote by $C^{k,\ell}(\overline{Q})$ the set of
 functions defined on $\overline{Q}$ with continuous partial derivatives
 up to order $k$ in
 the spatial  variable and up to order $\ell$ in
 time.
   Let $C^k(\overline{\Omega})$ be the set of
functions  defined on $\overline{\Omega}$ with
 continuous partial derivatives
up to order $k$.  Furthermore, for  $\delta\in (0, 1)$, we use
$ C^{k+\delta}(\overline{\Omega})$ and
$C^{k+\delta, \frac{k+\delta}{2}}(\overline{Q})$ to denote the corresponding
H\"older spaces defined on $\overline{\Omega}$ and $\overline{Q}$, respectively.
Let $C_0^{k+\delta}(\Omega)$ denote the subspace
 of  functions in $C^{k+\delta}(\overline{\Omega})$ with compact support in $\Omega$.
Recall that $C^{k+\delta, \frac{k+\delta}{2}}(\overline{Q})$ is a Banach space
 equipped  with the norm
$|h|_{C^{k+\delta, \frac{k+\delta}{2}}(\overline{Q})}=
\sum\limits_{0\leq 2r+|\sigma|\leq k} \sup\limits_{(x, t)\in\overline{Q}}|\partial_x^\sigma\partial_t^r h(x, t)|+
[h]_{k, \delta},$
where
$$
[h]_{k, \delta}\!=\!
\left\{\!
\begin{array}{ll}
\displaystyle\sum\limits_{2r+|\sigma|=k} \!
\sup\limits_{(x_1,t_1)\neq(x_2,t_2)}\!
\frac
{|\partial^\sigma_x\partial^{r}_t
	h(x_1,t_1)\!-\!\partial^\sigma_x\partial_t^{r}
	h(x_2,t_2)|}{(|x_1-x_2|\!+\!|t_1-t_2|^{1/2})^\delta},\  \ \  \  \mbox{ if } k \mbox{ is even};&\\[5mm]
	\displaystyle\sum\limits_{2r+|\sigma|=k}
	\sup\limits_{(x_1,t_1)\neq(x_2,t_2)}
	\frac
	{|\partial^\sigma_x\partial^{r}_t
		h(x_1,t_1)-\partial^\sigma_x\partial_t^{r}
		h(x_2,t_2)|}{(|x_1-x_2|+|t_1-t_2|^{1/2})^\delta}&\\
	\displaystyle \quad+
	\sum\limits_{2r+|\sigma|=k-1} \sup\limits_{x\in\overline{\Omega}}
	\sup\limits_{t_1\neq t_2}
	\frac
	{|\partial^\sigma_x\partial^{r}_t
		h(x,t_1)-\partial^\sigma_x\partial_t^{r}
		h(x,t_2)|}{|t_1-t_2|^{\frac{1+\delta}{2}}},\  \mbox{ if }k \mbox{ is odd},&
		\end{array}
		\right.
$$
with $\sigma=(\sigma_1, \cdots, \sigma_n)$ being a multi-index and $|\sigma|=\sigma_1+\cdots+\sigma_n$ (for detailed  definitions, see  \cite{La}). For any $k, \ell\in\dbN$ and $p\geq 2$, we use $W^{k, \ell}_p(Q)$ to denote the standard Sobolev space.

We require the following assumption:
 $$
\displaystyle{\bf (H) }\quad \quad
  f(0, {\bf 0})=0\quad\quad\mbox{and}\quad\quad \nabla f(0, {\bf 0})={\bf 0}.
$$
First,   we state the
 relationship between   (\ref{aeq1})
 and (\ref{aeq1*}) under identical  controls and initial values.
\begin{proposition}\label{th1}
Assume that ${\bf (H)}$ holds. Then there exists a positive constant $\rho_1$
such that for any  $y_0\in C^{3+\delta}(\overline{\Omega})$
and $u\in C^{1+\delta, \frac{1+\delta}{2}}(\overline{Q})$
satisfying the  compatibility conditions and
 $|y_0|_{C^{3+\delta}(\overline{\Omega})}+
|u|_{C^{1+\delta, \frac{1+\delta}{2}}(\overline{Q})}\leq \rho_1$,  the corresponding solutions
$y$ and $w$ to $(\ref{aeq1})$ and $(\ref{aeq1*})$ with $v=u$ satisfy
$$
|y-w|_{C^{3+\delta, \frac{3+\delta}{2}}(\overline{Q})}
\leq C\left(|y_0|^2_{C^{3+\delta}(\overline{\Omega})}+
|u|^2_{C^{1+\delta, \frac{1+\delta}{2}}(\overline{Q})}\right).
$$
Here and in what follows,
  $C$ denotes a generic
    positive constant depending only on $n, T, \Omega, \omega, \omega_0,
\delta$, $f$ and $a^{i j}$, which may vary from line to line.
\end{proposition}

\Cref{th1} demonstrates
 that the quasi-linear parabolic system \eqref{aeq1} can be accurately
 approximated by the associated linear system \eqref{aeq1}.
 More precisely, when $y_0$ and $u$ are sufficiently small,
 the difference between their corresponding solutions is a higher-order infinitesimal with respect to the data. The proof is deferred to the \Cref{chapter4}.
Importantly, this approximation property extends to the null controllability of the two equations, which constitutes the main result of this paper:
\begin{theorem}\label{th2}
Assume that ${\bf (H)}$ holds. Then there exists a positive constant $\rho_2$  such that
 for any initial value $y_0\in C^{3+\delta}(\overline\Omega)$  satisfying the  compatibility conditions and  $|y_0|_{C^{3+\delta}(\overline{\Omega})}\leq \rho_2$, one can  find
 controls $u^*, v^*\in C^{1+\delta, \frac{1+\delta}{2}}(\overline{Q})$ satisfying
 \begin{align*}
 	|u^*|_{C^{1+\delta, \frac{1+\delta}{2}}(\overline{Q})} + |v^*|_{C^{1+\delta, \frac{1+\delta}{2}}(\overline{Q})}
 	\leq C |y_0|_{C^{3+\delta}(\overline{\Omega})},
 \end{align*}
 so that the  corresponding solutions
$y$ and $w$ to $(\ref{aeq1})$ and $(\ref{aeq1*})$ satisfy
$$y(\cdot, T; u^*)=w(\cdot, T; v^*)=0 \text{ in }\Omega.$$
Moreover, the difference of the controls and the associated solutions satisfy the
following quadratic error estimate:
\begin{align}\label{approx_est}
	|u^*-v^*|_{C^{1+\delta, \frac{1+\delta}{2}}(\overline{Q})} +
	|y-w|_{C^{3+\delta,  \frac{3+\delta}{2}}(\overline{Q})}\leq C|y_0|^{2}_{C^{3+\delta}(\overline{\Omega})}.
\end{align}
\end{theorem}

Note that the null controllability results for the quasi-linear and linear parabolic equations
 stated in \Cref{th2} can be deduced from \cite{DLZ}; see Remark 1.5 and the proofs of Proposition 4.1 and Theorem 1.2 therein.
 However, the approximation estimate \eqref{approx_est} does not follow directly from \cite{DLZ}.
 The principal novelty of \Cref{th2} therefore lies in establishing this quadratic error estimate.
More precisely, the theorem guarantees the existence of controls such that the discrepancies in the corresponding states and controls are of quadratic order with respect to the initial data.
 Consequently, this provides  rigorous theoretical justification for approximating the quasi-linear controllability problem by its linearization in the small-data regime, thereby offering a quantitative refinement that surpasses the mere establishment of local null controllability.

We extend the approximation results  of \cite{LX} from  semi-linear equations to the present quasi-linear
setting,  featuring
 an improved estimate for the difference between the controls.
 This yields the error estimate of order $2$ in \eqref{approx_est},
 replacing  the order $3/2$ estimate  obtained in \cite[Theorem 1.1]{LX}.
	The article \cite{LX} investigated the relationship between the controllability of
	semi-linear and linear parabolic equations. In the present paper,
	we extend and improve this analysis to the more general
	and physically natural quasi-linear setting, in which
	 the principal diffusion coefficients depend on the state and its gradient.
	 In contrast to the semi-linear case, the approximation considered here
	  is not obtained merely by neglecting a nonlinear lower-order term; rather,
	  it requires replacing the nonlinear principal part by a constant-coefficient
	  linear operator, thereby  introducing new analytical difficulties.

\smallskip

There are two complementary methods  to obtain this approximation,
serving distinct functional   roles within the proof architecture.

\smallskip

{\it Modular Time-Splitting Route:}   This route provides the explicit,
global architecture of  transferring  from the linear system to the quasi-linear system.
Once Proposition \ref{th1} and  local
controllability results for the two systems are established,
 the control-level estimate follows in
three stages.
On $[0, T/3]$,  a null control for the linear system is  applied to the quasi-linear
system; Proposition \ref{th1} ensures
 the two states agree up to quadratic order at $T/3$. On $[T/3, 2T/3]$,
both controls are set to zero, allowing
 the linear state to remain  at zero while the nonlinear state evolves from
a quadratically small value.
On $[2T/3, T]$,  the linear control remains zero,
 and a small null control is
applied strictly  to the quasi-linear system.  Because
 both the state and the final control are already of
quadratic size,  the global control and state differences retain this order.
The central zero-control interval
 provides the necessary  regularity bridge  to
  seamlessly  concatenate the controls.
Thus, this route is
transparent and portable: it applies whenever same-control stability and local null controllability
are available. This approach  is developed in \Cref{chapter4}.
	
	\smallskip

{\it Direct  Linearized Route:}  This route supplies  the underlying
 analytic mechanism of
the transfer principle and does not require {\it a  prior}  controllability results
  for the quasi-linear
equation.   We first analyze the relationship between the controllability of the linearized system associated with \eqref{aeq1} and that of the linear equation \eqref{aeq1*},  deriving
 uniform
 higher-order estimates for the differences between their corresponding controls and states.
	A key ingredient in this analysis is a novel  Carleman estimate for linear parabolic equations.
	 Unlike the standard weights used in existing Carleman estimates (see \cite{Liux}),
	 our weight function is non-singular at $t=0$,  enabling
	  a direct  estimate of  the initial data for the adjoint equation.
	This type of weight has also been utilized
	 in  stochastic setting (e.g.,
	\cite{H, zxl}).  However, while
	 \cite{zxl} established   Carleman estimate for stochastic parabolic operators
	  under the assumption that the principal coefficients possess
	bounded second-order mixed space-time derivatives,
	 our  estimate  applies to
	  deterministic parabolic operators
	whose principal coefficients only satisfy  $a_{ij}\in C^{1,1}(\overline Q)$.
	This route is developed in \Cref{chapter2} and \Cref{chapter3}, providing  a detailed decomposition of the error into coefficients and initial-data contributions.

\smallskip

 These  two routes  form a unified   proof architecture. The time-splitting argument outlines
the abstract transfer principle in a concise and elegant manner.
 Consequently,  it is broadly applicable to
 other nonlinear evolution equations where
  the corresponding controllability and same-control approximation results are known.
Conversely,
this direct method provides
 structural insights into   how the nonlinearity affects the approximation problem.
 In particular, it explicitly demonstrates
  that the quadratic error estimate between the controls for the quasi-linear and the approximate linear system arises in equal  part from the initial values, and  from
   the principal coefficients $a^{ij}$ and nonlinear function $f$.
   It
    thereby yields  finer structural information
    regarding  on the control construction and  error decomposition.

\smallskip

Furthermore,  assume that $\tilde{u}\in C^{1+\delta, \frac{1+\delta}{2}}(\overline{Q})$ and
$\tilde{y}_0\in C_0^{3+\delta}(\Omega)$, and that  $\tilde{y}, \tilde{w}\in C^{3+\delta, \frac{3+\delta}{2}}(\overline{Q})$ satisfy, respectively,
the following quasi-linear parabolic equation:
\begin{eqnarray}\label{aeq1***}
	\left\{
	\begin{array}{ll}
		\tilde{y}_t-\sum\limits_{i, j=1}^n \left(a^{ij}(\tilde{y}, \nabla \tilde{y})\tilde{y}_{x_i}\right)_{x_j} + f(\tilde{y}, \nabla \tilde{y})=\xi_0 \tilde{u} &\mbox{ in }Q, \\ \ns
		\tilde{y}=0 &\mbox{ on }\Sigma, \\ \ns
		\tilde{y}(x, 0)=\tilde{y}_0(x) &\mbox{ in }\Omega,
	\end{array}
	\right.
\end{eqnarray}
and linear parabolic equation:
\begin{eqnarray}\label{aeq1****}
	\left\{
	\begin{array}{ll}
		\tilde{w}_t-\sum\limits_{i, j=1}^n \left(\tilde{a}^{ij}\tilde{w}_{x_i}\right)_{x_j}=\xi_0 \tilde{u} &\mbox{ in }Q, \\ \ns
		\tilde{w}=0 &\mbox{ on }\Sigma, \\ \ns
		\tilde{w}(x, 0)=\tilde{y}_0(x) &\mbox{ in }\Omega.
	\end{array}
	\right.
\end{eqnarray}
Then, as an immediate corollary to Theorem \ref{th1}, we have the following approximation result for the controllability of  trajectories.
\begin{corollary}\label{coro}
	Assume that ${\bf (H)}$ holds. Then there exists a positive constant $\rho_3$  such that
	for any  $y_0, \tilde{y}_0\in C^{3+\delta}(\overline\Omega)$  and $\tilde{u}\in C^{1+\delta, \frac{1+\delta}{2}}(\overline{Q})$ satisfying the  compatibility conditions  and
	$$|y_0|_{C^{3+\delta}(\overline{\Omega})}
	+|\tilde{y}_0|_{C^{3+\delta}(\overline{\Omega})}
	+|\tilde{u}|_{C^{1+\delta, \frac{1+\delta}{2}}(\overline{Q})}\leq \rho_3,$$
	one can  find
	controls $u,  v\in C^{1+\delta, \frac{1+\delta}{2}}(\overline{Q})$ so that the  corresponding solutions
	$y$ and $w$ to $(\ref{aeq1})$ and $(\ref{aeq1*})$ satisfy
	$$
	y(\cdot, T; u)=\tilde{y}(\cdot, T)\quad\mbox{ and }\quad w(\cdot, T; v)=\tilde{w}(\cdot, T)\quad
	\mbox{ in } \Omega.$$
	Moreover,
	$$
	|u-v|_{C^{1+\delta, \frac{1+\delta}{2}}(\overline{Q})}+
	|y-w|_{C^{3+\delta, \frac{3+\delta}{2}}(\overline{Q})}\leq
	C\Big(|y_0|^{2}_{C^{3+\delta}(\overline{\Omega})}+
	|\tilde{y}_0|^{2}_{C^{3+\delta}(\overline{\Omega})}+|\tilde{u}|^2_{C^{1+\delta, \frac{1+\delta}{2}}(\overline{Q})}\Big).
	$$
\end{corollary}

The remainder  of this paper is organized as follows.   \Cref{chapter4} is devoted to providing  the
indirect method to prove Theorem   \ref{th2}.  \Cref{chapter2}  establishes  global  Carleman
estimates for linear parabolic equations with $C^{1, 1}$-principal part coefficients.  \Cref{chapter3} presents the direct method for  the alternate  proof of \Cref{th2}.

\section{The indirect method for  controllability approximation}\label{chapter4}
In this section, we provide  an indirect
 proof of Theorem \ref{th2}.
  The proof relies on the  local
   null controllability results for the quasi-linear system \eqref{aeq1}
    and the linear system \eqref{aeq1*},  alongside  the approximation result given
     in  \Cref{th1}.

    First, we give a proof of Proposition \ref{th1}.
    \begin{proof}[Proof of \Cref{th1}]
	Let $y$ and $w$ be the  solutions to the systems (\ref{aeq1}) and (\ref{aeq1*}) corresponding to $v=u$,  respectively.
	Set $z=y-w$ and $\widehat{a}^{i j}=a^{i j}-\tilde{a}^{i j}$. Then $z$ satisfies
	\begin{eqnarray*}
		\left\{
		\begin{array}{ll}
			z_t-\sum\limits_{i, j=1}^{n} (\tilde{a}^{i j} z_{x_i})_{x_j}
			=\sum\limits_{i, j=1}^{n} (\widehat{a}^{i j}(y, \nabla y) y_{x_i})_{x_j}-f(y, \nabla y) &\mbox{ in }Q, \\ \ns
			z=0 &\mbox{ on }\Sigma, \\ \ns
			z(x, 0)=0 &\mbox{ in }\Omega.\\ \ns
		\end{array}
		\right.
	\end{eqnarray*}
	
	By the well-posedness results  for linear parabolic equations (\cite[Theorem 5.2 on Page 320]{La} and \cite[Theorem 4.28 on Page 77]{Li}),  we have
	\begin{eqnarray*}
		&&|y-w|_{C^{3+\delta, \frac{3+\delta}{2}}(\overline{Q})}=|z|_{C^{3+\delta, \frac{3+\delta}{2}}(\overline{Q})}\\
		&&\leq C\big|\sum\limits_{i, j=1}^{n} (\widehat{a}^{i j}(y, \nabla y) y_{x_i})_{x_j} \big|_{C^{1+\delta, \frac{1+\delta}{2}}(\overline{Q})}
		+C\big|f(y, \nabla y)\big|_{C^{1+\delta, \frac{1+\delta}{2}}(\overline{Q})}.
	\end{eqnarray*}

	By the definition of H\"older's space,  it holds that
	$$
	|h_1 h_2|_{C^{1+\delta, \frac{1+\delta}{2}}(\overline{Q})}\leq C|h_1|_{C^{1+\delta, \frac{1+\delta}{2}}(\overline{Q})}
	|h_2|_{C^{1+\delta, \frac{1+\delta}{2}}(\overline{Q})}, \quad\forall\ h_1, h_2\in C^{1+\delta, \frac{1+\delta}{2}}(\overline{Q}).
	$$
	Noting that $f(0, {\bf 0})=0$,  $\nabla f(0, {\bf 0})={\bf 0}$ and $\widehat{a}^{i j}(0, {\bf 0})=0$ $(i, j=1, \cdots, n)$,
	we deduce
	
	\begin{eqnarray*}
		&&|y-w|_{C^{3+\delta, \frac{3+\delta}{2}}(\overline{Q})}\\
		&&\leq C\sum\limits_{i, j=1}^{n}
		\left[\big|\widehat{a}^{i j}(y, \nabla y) \big|_{C^{1+\delta, \frac{1+\delta}{2}}(\overline{Q})}
		\big|y_{x_i x_j} \big|_{C^{1+\delta, \frac{1+\delta}{2}}(\overline{Q})}
		+\big|\big(\widehat{a}^{i j}(y, \nabla y)\big)_{x_j} \big|_{C^{1+\delta, \frac{1+\delta}{2}}(\overline{Q})}
		\big|y_{x_i} \big|_{C^{1+\delta, \frac{1+\delta}{2}}(\overline{Q})}\right]\\
		&&\quad+C\bigg[
		\Big|\int^1_0 f_s(\tau y, \tau \nabla y)d\tau\Big|_{C^{1+\delta, \frac{1+\delta}{2}}(\overline{Q})}
		|y|_{C^{1+\delta, \frac{1+\delta}{2}}(\overline{Q})}\\
		&&\quad\quad\quad\quad
		+\Big|\int^1_0 \nabla_\zeta f(\tau y, \tau \nabla y)d\tau\Big|_{C^{1+\delta, \frac{1+\delta}{2}}(\overline{Q})}
		|\nabla y|_{C^{1+\delta, \frac{1+\delta}{2}}(\overline{Q})}\bigg]
	\end{eqnarray*}
	\begin{eqnarray*}
		&&\leq C|y|_{C^{3+\delta, \frac{3+\delta}{2}}(\overline{Q})}
		\Bigg\{\sum\limits_{i, j=1}^{n} \big|\widehat{a}^{i j}(y, \nabla y)-\widehat{a}^{i j}(0, {\bf 0}) \big|_{C^{1+\delta, \frac{1+\delta}{2}}(\overline{Q})}
		+\sum\limits_{i, j=1}^{n}\big|\widehat{a}^{i j}_s(y, \nabla y)y_{x_j} \big|_{C^{1+\delta, \frac{1+\delta}{2}}(\overline{Q})}\\
		&&\quad\quad\quad+
		\sum\limits_{i, j=1}^{n}\big|\nabla_{\zeta}\widehat{a}^{i j}(y, \nabla y)\cdot \nabla y_{x_j} \big|_{C^{1+\delta, \frac{1+\delta}{2}}(\overline{Q})}
		+\Big|\int^1_0 [f_s(\tau y, \tau \nabla y)-f_s(0, {\bf 0})] d\tau\Big|_{C^{1+\delta, \frac{1+\delta}{2}}(\overline{Q})}\\
		&&\quad\quad\quad+\Big|\int^1_0 [\nabla_\zeta f(\tau y, \tau \nabla y)- \nabla_{\zeta}
		f(0, {\bf 0})]d\tau\Big|_{C^{1+\delta, \frac{1+\delta}{2}}(\overline{Q})}\Bigg\}\\
		&&\leq  C|y|^2_{C^{3+\delta, \frac{3+\delta}{2}}(\overline{Q})}.
	\end{eqnarray*}

	Furthermore, by the local well-posedness results  for the quasi-linear parabolic equation (\ref{aeq1}),  it follows that
	$$
	|y|_{C^{3+\delta, \frac{3+\delta}{2}}(\overline{Q})}\leq
	C\Big(|y_0|_{C^{3+\delta}(\overline{\Omega})}+|u|_{C^{1+\delta, \frac{1+\delta}{2}}(\overline{Q})}\Big).
	$$
	Therefore,
	$$
	|y-w|_{C^{3+\delta, \frac{3+\delta}{2}}(\overline{Q})}
	\leq  C\Big(|y_0|^2_{C^{3+\delta}(\overline{\Omega})}+|u|^2_{C^{1+\delta, \frac{1+\delta}{2}}(\overline{Q})}\Big).
	$$
	This completes the proof of  Proposition \ref{th1}.
\end{proof}

Next,   we recall the controllability results for the quasi-linear and linear parabolic equations.
By    \cite[Remark 1.5]{DLZ} and the proofs of   \cite[Proposition 4.1 and Theorem 1.2]{DLZ}, we have the following
 local null controllability result for
the quasi-linear parabolic system  \eqref{aeq1}  and the linear parabolic system  \eqref{aeq1*}, with  controls in $C^{1+\delta,\frac{1+\delta}{2}}(\overline Q)$ for $\delta\in(0,1)$.
\begin{lemma}\label{lemma1!}
There exists a positive constant $\rho^*$ such that for any  $y_0\in C^{3+\delta}_0(\Omega)$ with
$|y_0|_{C^{3+\delta}(\overline{\Omega})}\leq \rho^*$,  one can find two controls
$u^*, v^*\in C^{1+\delta, \frac{1+\delta}{2}}(\overline{Q})$ so that the corresponding solutions $y$ and $w$
 to $(\ref{aeq1})$  and $(\ref{aeq1*})$
satisfy $y(\cdot, T; u^*)=w(\cdot, T; v^*)=0$ in $\Omega$. Moreover,
$$
|u^*|_{C^{1+\delta, \frac{1+\delta}{2}}(\overline{Q})}+|v^*|_{C^{1+\delta, \frac{1+\delta}{2}}(\overline{Q})}\leq C|y_0|_{L^2(\Omega)}.
$$
\end{lemma}

\begin{remark}\label{remark1}
The requirement $y_0\in C^{3+\delta}_0(\Omega)$ can be weakened  to
$y_0\in C^{3+\delta}(\overline{\Omega})$ with
 the  compatibility conditions needed  for the well-posedness of the system
$(\ref{aeq1})$  in the space  $C^{3+\delta, \frac{3+\delta}{2}}(\overline{Q})$.
\end{remark}
\begin{remark}\label{remark2}
The controls $u^*$ $($similarly, $v^*)$ in Lemma $\ref{lemma1!}$ is  indeed
  the limit of a sequence of uniformly bounded functions $u_{\varepsilon}$
 $($similarly, $v_\varepsilon)$ in $C^{1+\delta, \frac{1+\delta}{2}}(\overline{Q})$.
 This sequence $\{u_{\varepsilon}\}_{\varepsilon>0}$
  is the product of  a $C^{1+\delta, \frac{1+\delta}{2}}(\overline{Q})$-function $m_\varepsilon^{N^*}$ and  an infinitely differentiable function $e^{2\lambda\nu-(\lambda+\delta_{N^*})\nu_0}\beta^7 \beta_0^{-7}\xi_0$ $($which can be found on Page $3799$ in $\cite{DLZ})$, whose derivatives  of all orders decay exponentially to zero near time $0$  and $T$.
Therefore, when the control functions $u_\varepsilon$ and  $v_\varepsilon$
 are extended by zero  to  the time interval $[-\tau,  T+\tau]$ for a $\tau>0$ $($still denoted by the same notations$)$,
they still remain uniformly bounded  in the $C^{1+\delta, \frac{1+\delta}{2}}$-space, i.e.,
\begin{eqnarray*}
&&|u_\varepsilon|_{C^{1+\delta, \frac{1+\delta}{2}}(\overline{\Omega}\times [-\tau, T+\tau])}+
|v_\varepsilon|_{C^{1+\delta, \frac{1+\delta}{2}}(\overline{\Omega}\times [-\tau, T+\tau])}\\
&&\leq C(|u_\varepsilon|_{C^{1+\delta, \frac{1+\delta}{2}}(\overline{\Omega}\times [0, T])}
+|v_\varepsilon|_{C^{1+\delta, \frac{1+\delta}{2}}(\overline{\Omega}\times [0, T])}) \leq C|y_0|_{L^2(\Omega)}.
\end{eqnarray*}
Consequently,
 there exist controls $u^*$ and $v^*$ in
 $C^{1+\delta, \frac{1+\delta}{2}}(\overline{\Omega}\times [-\tau, T+\tau])$ that
 vanish  on  $\Omega\times([-\tau, 0]\cup[T, T+\tau])$ such that
$$u_{\varepsilon}\rightarrow u^*\quad\mbox{and}\quad v_{\varepsilon}\rightarrow v^*\quad
\mbox{in }\quad C^{1+\delta, \frac{1+\delta}{2}}(\overline{\Omega}\times [-\tau, T+\tau]).$$
\end{remark}

Now, we give a proof of Theorem \ref{th2}.

\begin{proof}[Proof of \Cref{th2}]
	We have already concluded about the null controllability for the control systems \eqref{aeq1} and \eqref{aeq1*}. Therefore, in order to establish the \Cref{th2}, we only need to prove the error estimate \eqref{approx_est}.
	For that, we divide the  time interval $[0, T]$ into three segments.

	\medskip
	
	{\bf Step 1. }  The time  interval  $ [0, T/3]$.
	
	\medskip
	
	First,
	by the null controllability result for  the linear parabolic equation (\ref{aeq1*})   in Lemma
	\ref{lemma1},  there exists a control
	$v_1\in C^{1+\delta, \frac{1+\delta}{2}}(\overline{\Omega}\times[0, T/3])$ such that
	the corresponding solution  $w$
	to $(\ref{aeq1*})$
	satisfies $w(\cdot, T/3; v_1)=0$ in $\Omega$, with
	\begin{equation}\label{ena1}
		|v_1|_{C^{1+\delta, \frac{1+\delta}{2}}(\overline{\Omega}\times[0, T/3])}\leq C|y_0|_{L^2(\Omega)}.
	\end{equation}

	Next, for the quasi-linear parabolic equation (\ref{aeq1}),
	we choose $u=v_1$ on $\Omega\times[0, T/3]$.  By  (\ref{ena1}),  there exists a  $\rho_2>0$,  such that
	\begin{eqnarray*}
		|y_0|_{C^{3+\delta}(\overline{\Omega})}+|u|_{C^{1+\delta, \frac{1+\delta}{2}}(\overline{\Omega}\times[0, T/3])}=
		|y_0|_{C^{3+\delta}(\overline{\Omega})}+|v_1|_{C^{1+\delta, \frac{1+\delta}{2}}(\overline{\Omega}\times[0, T/3])}\leq C\rho_2\leq \rho_1.
	\end{eqnarray*}
	Then,  by Proposition \ref{th1} and (\ref{ena1}), the corresponding solutions
	$y$ and $w$ to $(\ref{aeq1})$ and $(\ref{aeq1*})$ with the same control $u=v=v_1$ on $\Omega\times[0, T/3]$ satisfy
	\begin{equation}\label{ena2}
		|y-w|_{C^{3+\delta, \frac{3+\delta}{2}}(\overline{\Omega}\times[0, T/3])}
		\leq C\Big(|y_0|^2_{C^{3+\delta}(\overline{\Omega})}+
		|v_1|^2_{C^{1+\delta, \frac{1+\delta}{2}}(\overline{\Omega}\times[0, T/3])}\Big)\leq C
		|y_0|^2_{C^{3+\delta}(\overline{\Omega})}.
	\end{equation}
	This  implies that
	\begin{equation}\label{ena3}
		|y(\cdot, T/3)|_{C^{3+\delta}(\overline{\Omega})}=
		|y(\cdot, T/3)-w(\cdot, T/3)|_{C^{3+\delta}(\overline{\Omega})}\leq C
		|y_0|^2_{C^{3+\delta}(\overline{\Omega})},
	\end{equation}
	and
	\begin{equation}\label{ena3*}
		|y(\cdot, T/3)|_{C^{3+\delta}(\overline{\Omega})}\leq
		C|y_0|_{C^{3+\delta}(\overline{\Omega})}\leq C\rho_2.
	\end{equation}

	{\bf Step 2. }The time interval  $[T/3,  2T/3]$.
	
	\medskip

	On this interval,
	we set  the controls $u$ and $v$ to zero for both
	the quasi-linear and linear parabolic equations. By Remark \ref{remark2}, for the control $v_1\in C^{1+\delta, \frac{1+\delta}{2}}(\overline{\Omega}\times[0, T/3])$, the zero-extended control
	belongs to  $C^{1+\delta, \frac{1+\delta}{2}}(\overline{\Omega}\times[0, 2T/3])$.

	Next, for the quasi-linear parabolic equation (\ref{aeq1}),
	since $y_0\in C^{3+\delta}_0(\Omega)$ and $v_1\in C^{1+\delta, \frac{1+\delta}{2}}(\overline{\Omega}\times[0, T/3])$ with supp$(\xi_0 v_1)\subseteq \Omega\times [0, T/3]$,
	the associated solution $y$ to
	(\ref{aeq1}) belongs to  $C^{3+\delta, \frac{3+\delta}{2}}(\overline{\Omega}\times[0, T/3])$
	and hence,
	$y(\cdot, T/3)\in C^{3+\delta}(\overline{\Omega})$.  Moreover,
	as the initial value of the solution to
	(\ref{aeq1}) on $\Omega\times[T/3, 2T/3]$, it also satisfies the the  compatibility conditions, since
	the support of the non-homogenous term  is always contained in $\Omega\times[0, T]$.
	
	By the local well-posedness result for  quasi-linear parabolic equations, together with (\ref{ena3}) and (\ref{ena3*}),
	we obtain
	\begin{equation}\label{ena3!}
		|y|_{C^{3+\delta, \frac{3+\delta}{2}}(\overline{\Omega}\times[T/3, 2T/3])}+|y(\cdot, 2T/3)|_{C^{3+\delta}(\overline{\Omega})}\leq C|y(\cdot, T/3)|_{C^{3+\delta}(\overline{\Omega})}\leq C
		|y_0|^2_{C^{3+\delta}(\overline{\Omega})},
	\end{equation}
	and
	\begin{equation}\label{ena3*!}
		|y(\cdot, 2T/3)|_{C^{3+\delta}(\overline{\Omega})}\leq
		C\rho_2.
	\end{equation}

	\medskip

	{\bf Step 3. } The  time interval  $[2T/3, T]$.
	
	\medskip

	First,  for the linear parabolic equation (\ref{aeq1*}), we
	choose $v=0$ on $\Omega\times[2T/3, T]$.
	Since $w(\cdot, T/3)=0$ in $\Omega$, the associated solution $w$ satisfies $w=0$ on $\Omega\times[T/3, T]$, and thus, $w(\cdot, T)=0$ in $\Omega$.
	
	\smallskip

	Next,  by (\ref{ena3*!}), and the local null controllability result for the quasi-linear parabolic equation (\ref{aeq1}) in Lemma \ref{lemma1},   for a sufficiently small
	$\rho_2$,
	there exists a control $u_2\in
	C^{1+\delta, \frac{1+\delta}{2}}(\overline{\Omega}\times[2T/3, T])$, such that  the associated solution $y$ to (\ref{aeq1}) satisfies
	$y\in C^{3+\delta, \frac{3+\delta}{2}}(\overline{\Omega}\times[2T/3, T])$ and
	$y(\cdot, T)=0$ in $\Omega$. Moreover,  by (\ref{ena3!}), it holds that
	\begin{equation}\label{ena4}
		|u_2|_{C^{1+\delta, \frac{1+\delta}{2}}(\overline{\Omega}\times[2T/3, T])}
		\leq C|y(\cdot,  2T/3)|_{L^2(\Omega)}\leq C|y_0|^2_{C^{3+\delta}(\overline{\Omega})}.
	\end{equation}
	From (\ref{ena3!}) and (\ref{ena4}),  the corresponding solutions
	$y$ and $w$ to $(\ref{aeq1})$ and $(\ref{aeq1*})$  satisfy
	\begin{eqnarray}\label{ena5}
		\begin{array}{rl}
			&\displaystyle|y-w|_{C^{3+\delta, \frac{3+\delta}{2}}(\overline{\Omega}\times[2T/3, T])}=
			|y|_{C^{3+\delta, \frac{3+\delta}{2}}(\overline{\Omega}\times[2T/3, T])}\\[2mm]
			&\displaystyle\leq C\Big(|y(\cdot, 2T/3)|_{C^{3+\delta}(\overline{\Omega})}+
			|u_2|_{C^{1+\delta, \frac{1+\delta}{2}}(\overline{\Omega}\times[2T/3, T])}\Big)\leq
			C|y_0|^2_{C^{3+\delta}(\overline{\Omega})}.
		\end{array}
	\end{eqnarray}
	
	Moreover, by Remark \ref{remark2}, for the control $u_2\in C^{1+\delta, \frac{1+\delta}{2}}(\overline{\Omega}\times[2T/3, T])$, the zero-extended control
	belongs to the space $C^{1+\delta, \frac{1+\delta}{2}}(\overline{\Omega}\times[T/3, T])$.

	\medskip
	
	{\bf Step 4. }
	In conclusion, we define the global controls as
	$$
	u^*(x, t)=\left\{
	\begin{array}{ll}
		v_1 &(x, t)\in \Omega\times[0, T/3],\\
		0 &(x, t)\in \Omega\times[T/3, 2T/3],\\
		u_2 &(x, t)\in \Omega\times [2T/3, T],
	\end{array}
	\right.
	\quad
	v^*(x, t)=\left\{
	\begin{array}{ll}
		v_1 &(x, t)\in \Omega\times[0, T/3],\\
		0 &(x, t)\in \Omega\times [T/3, T].
	\end{array}
	\right.
	$$
By Remark \ref{remark2},
	$u^*, v^*\in C^{1+\delta, \frac{1+\delta}{2}}(\overline{\Omega}\times[0, 2T/3])\cap
	C^{1+\delta, \frac{1+\delta}{2}}(\overline{\Omega}\times[T/3, T])$, which indicates
	$u^*, v^*\in C^{1+\delta, \frac{1+\delta}{2}}(\overline{Q})$, satisfying
	\begin{eqnarray*}
		&&|u^*|_{C^{1+\delta, \frac{1+\delta}{2}}(\overline{Q})}\leq C\big(|v_1|_{C^{1+\delta, \frac{1+\delta}{2}}(\overline{\Omega}\times[0, T/3])}+|u_2|_{C^{1+\delta, \frac{1+\delta}{2}}(\overline{\Omega}\times[2T/3, T])}\big),\\
		&&\mbox{and }\ |v^*|_{C^{1+\delta, \frac{1+\delta}{2}}(\overline{Q})}\leq C
		|v_1|_{C^{1+\delta, \frac{1+\delta}{2}}(\overline{\Omega}\times[0, T/3])}.
	\end{eqnarray*}
	Therefore,   the corresponding solutions
	$y$ and $w$ to $(\ref{aeq1})$ and $(\ref{aeq1*})$  satisfy $y,  w\in C^{3+\delta, \frac{3+\delta}{2}}(\overline{Q})$. Moreover,
	$$
	y(\cdot, T; u^*)=w(\cdot, T; v^*)=0\quad\mbox{ in }\Omega.
	$$
	Furthermore, (\ref{ena2}), (\ref{ena3!}), (\ref{ena4}) and (\ref{ena5}) imply
	\begin{eqnarray*}
		&&|y-w|_{C^{3+\delta, \frac{3+\delta}{2}}(\overline{Q})}\\
		&&\leq C
		\Big(|y-w|_{C^{3+\delta, \frac{3+\delta}{2}}(\overline{\Omega}\times[0, T/3])}
		+|y-w|_{C^{3+\delta, \frac{3+\delta}{2}}(\overline{\Omega}\times[T/3, 2T/3])}+
		|y-w|_{C^{3+\delta, \frac{3+\delta}{2}}(\overline{\Omega}\times[2T/3, T])}\Big)\\
		&&\leq C|y_0|^2_{C^{3+\delta}(\overline{\Omega})},
	\end{eqnarray*}
	and
	$$
	|u^*-v^*|_{C^{1+\delta, \frac{1+\delta}{2}}(\overline{Q})}\leq C|u_2|_{C^{1+\delta, \frac{1+\delta}{2}}(\overline{\Omega}\times[2T/3, T])}\leq C|y_0|^2_{C^{3+\delta}(\overline{\Omega})}.
	$$
	This finishes the proof of Theorem \ref{th2}.
\end{proof}

Finally, we give a proof of Corollary \ref{coro}.

\begin{proof}[Proof of \Cref{coro}]
First, we consider the differences between the quasi-linear parabolic  equations (\ref{aeq1}) and (\ref{aeq1***}), and between the linear parabolic equations (\ref{aeq1***}) and (\ref{aeq1****}), respectively.

We apply the null controllability approximation result from Theorem \ref{th2} to  the resulting difference equations. Then, there exists a $\rho_3>0$ such that for any initial value $y_0, \tilde{y}_0\in
C^{3+\delta}(\overline\Omega)$ satisfying the  compatibility conditions  and
  $|y_0-\tilde{y}_0|_{C^{3+\delta}(\overline{\Omega})}\leq \rho_3$, one can find two controls
 $u, v\in C^{1+\delta, \frac{1+\delta}{2}}(\overline{Q})$ so that
 $$
 y(\cdot, T; u)=\tilde{y}(\cdot, T)\quad \mbox{and} \quad
 w(\cdot, T; u)=\tilde{w}(\cdot, T)\quad\mbox{ in }\Omega.
 $$
Moreover,
$$
|u-v|_{C^{1+\delta, \frac{1+\delta}{2}}(\overline{Q})}
+|(y-\tilde{y})-(w-\tilde{w})|_{C^{3+\delta, \frac{3+\delta}{2}}(\overline{Q})}\leq C
|y_0-\tilde{y}_0|^2_{C^{3+\delta}(\overline{\Omega})},
$$
which, together with  Proposition \ref{th1},   indicates
\begin{eqnarray*}
&&|y-w|_{C^{3+\delta, \frac{3+\delta}{2}}(\overline{Q})}\leq
|\tilde{y}-\tilde{w}|_{C^{3+\delta, \frac{3+\delta}{2}}(\overline{Q})}+
C
|y_0-\tilde{y}_0|^2_{C^{3+\delta}(\overline{\Omega})}\\
&&\leq C\Big(|\tilde{y}_0|^2_{C^{3+\delta}(\overline{\Omega})}
+|\tilde{u}|^2_{C^{1+\delta, \frac{1+\delta}{2}}(\overline{Q})}
+|y_0-\tilde{y}_0|^2_{C^{3+\delta}(\overline{\Omega})}\Big)\\
&&\leq C\Big(|y_0|^2_{C^{3+\delta}(\overline{\Omega})}+
|\tilde{y}_0|^2_{C^{3+\delta}(\overline{\Omega})}
+|\tilde{u}|^2_{C^{1+\delta, \frac{1+\delta}{2}}(\overline{Q})}
\Big).
\end{eqnarray*}
This finishes the proof of Corollary \ref{coro}.
\end{proof}

\section{Carleman estimates for linear parabolic equations}\label{chapter2}
To establish the connection between the controllability of \eqref{aeq1} and \eqref{aeq1*} via a
direct method, we
 first develop global Carleman estimates for linear parabolic equations with $C^{1, 1}$ principal
 part coefficients. To this end, we consider the following  parabolic equation:
\begin{eqnarray}\label{e1}
	\left\{
	\begin{array}{ll}
		p_t+\sum\limits_{i, j=1}^n \left(b^{i j} p_{x_i}\right)_{x_j}
		=F   &\mbox{ in   }Q,\\[2mm]
		p=0  &\mbox{ on  }\Sigma,\\[2mm]
		p(x, T)=p_T(x)  &\mbox{ in  }\Omega,
	\end{array}
	\right.
\end{eqnarray}
where   $p_T\in L^2(\Omega)$,
$F\in L^2(Q)$,
and
$b^{i j}\in C^{1, 1}(\overline{Q})$ with
$b^{i j}=b^{j i}$ $(i, j=1, \cdots, n)$ satisfying
\begin{align}\label{b_ij_cond}
	\sum\limits_{i, j=1}^{n} b^{i j}(x, t)\eta_{i}\eta_{j}
	\geq \rho_0 |\eta|^2,\quad\quad\forall\ (x, t, \eta)=(x, t, \eta_1, \cdots, \eta_n)\in
	\overline{Q}\times\mathbb{R}^n,
\end{align}
for the positive constant $\rho_0$.



In order to establish Carleman estimates, let us first introduce some required weight functions.
By Lemma 5.1 in \cite{TZ}  and  arguments similar to those in \cite{FI, WW}, there exists a function $\psi\in C^4(\overline{\Omega})$ such that
\begin{align}\label{psi}
	0<\psi(x)\leq 1  \   \mbox{ in  }\Omega, \quad
	\psi(x)=0  \    \mbox{ on  }\Gamma,   \quad
	\mbox{ and  } \quad
	|\nabla\psi(x)|>0\  \mbox{ in }\overline{\Omega\setminus\tilde{\omega}_0},
\end{align}
where $\tilde{\omega}_0$ is an open subset  satisfying
$\overline{\tilde{\omega}_0}\subseteq \omega_0$.
Without loss of generality, assume that  $T<1$.
For any $\lambda, \mu\geq 1$ and  $\sigma_0=\lambda\mu^2 e^{2\mu}$, define the weight function $\gamma_0\in C^2([0, T))$ by
\begin{eqnarray}\label{gamma}
\gamma_0(t)=\left\{
\begin{array}{ll}
\displaystyle
1+\Big(1-\frac{4t}{T}\Big)^{\sigma_0} &t\in [0, T/4],\\
\displaystyle 1 &t\in [T/4, T/2],\\
\displaystyle\frac{1}{T-t} &t\in [3T/4,  T),
\end{array}
\right.
\end{eqnarray}
such that $\gamma_0$ is monotonically increasing on $[T/2,  3T/4]$,  and is independent of $\lambda$
and $\mu$ on this interval.
Note that the first-order derivative $\gamma_0'$ of it is given by
\begin{eqnarray}\label{gamma'_first}
	\gamma_0'(t)
	= \left\{
	\begin{array}{ll}
	\displaystyle -\frac{4}{T}\sigma_0 \Big(1-\frac{4t}{T}\Big)^{\sigma_0-1} &t\in [0,T/4], \\
		0 & t\in [T/4,T/2], \\
		\displaystyle\frac{1}{(T-t)^2} & t\in [3T/4,T),
	\end{array}\right.
\end{eqnarray}
 $\gamma_0'(t)\leq 0$ for $t\in [0,T/4]$,  and $\gamma_0'(t)\geq0$ for $t\in(T/4,T)$.
Moreover, there exists a positive constant $C$, independent of $\lambda$ and $\mu$, such that
on $[0, T)$,
\begin{align}\label{gamma_prop}
\begin{cases}
	\quad \gamma_0(t) \geq 1, &\\
	\quad \left|\gamma_0'(t)\right| \leq C \sigma_0\gamma_0^2(t) = C\lambda\mu^2 e^{2\mu}\gamma_0^2(t). &
\end{cases}
\end{align}
We introduce the  weight functions:
$$
\alpha(x, t)= \big\{e^{\mu[\psi(x)+6]}-
\mu e^{12\mu}\big\}\gamma_0(t),\quad
\varphi(x,t) = e^{\mu[\psi(x)+6]}\gamma_0(t)\quad\mbox{and}\quad
\theta(x, t)=e^{\lambda\alpha(x, t)},
$$
and set $\Lambda:= \sum\limits_{i, j=1}^{n}   |b^{i j}|^2_{C^{1, 1}(\overline{Q})}$.

Let us now state the  first  Carleman estimate for  the adjoint system \eqref{e1}.
\begin{proposition}\label{lemma1}
	There exist  positive constants $\mu_0$,   $\lambda_0$ and $C_0$ such that for any $\mu\geq \mu_0$ and $\lambda\geq \lambda_0$, any solution $p$ to $(\ref{e1})$  satisfies
\begin{eqnarray}\label{carleman-1}
\begin{array}{rl}
	&\displaystyle \int_\Omega
	\lambda^2\mu^3 e^{14\mu}
	\theta^2(x, 0) p^2(x, 0) dx
	+\int_Q \theta^2\Big(
	\lambda^3\mu^4\varphi^3 p^2+\lambda\mu^2\varphi |\nabla p|^2\Big) dxdt\\
	&\quad\displaystyle +\int_Q \theta^2 \lambda^{-1}\varphi^{-1}\Big(
	|p_t|^2+ \sum\limits_{i, j=1}^{n}\big|p_{x_i x_j}\big|^2\Big) dxdt\\
	&\displaystyle
	 \leq C_0
	\Big(\int^{T}_{0}\int_{\omega_0}\theta^2\lambda^3\mu^4 e^{2\mu}\varphi^3 p^2dxdt+
	\int_Q\theta^2 F^2dxdt\Big),
\end{array}
\end{eqnarray}
for all $p_T\in L^2(\Omega)$, where the constant $C_0$ depends on $\rho_0$,  $\Omega$,  $\omega_0$, $n$, $T$, $\lambda_0$, $\mu_0$ and $\Lambda$, but is independent of $\lambda$ and $\mu$.
\end{proposition}

\begin{proof}
	Throughout the proof, $C_0$ denotes a generic positive constant, which may change from line to line,  but is independent of the  parameters $\lambda$ and $\mu$.
	
	Following an argument similar to that in Proposition 2.1 of \cite{Liux}, we derive the desired Carleman estimate \eqref{carleman-1}. The main difference arises from the terms involving $p(\cdot,0)$ and the derivatives of $\gamma_0$, which are important for obtaining the estimate at the initial time. In particular, from \eqref{gamma'_first}, we have
	$0\leq \gamma_0'(t)\leq \gamma_0^2(t), \text{ for } t\in(T/4,T),$
	whereas in \cite{Liux}, the same estimate holds on the whole interval $(0,T)$. On the other hand,  we have $\gamma_0'(t)\leq 0$, for $t\in[0,T/4]$.
	Thus, the argument in \cite{Liux} needs to be modified for $t\in[0,T/4]$.

	First, let $q=e^{\lambda\alpha}p$ and  $\ell=\lambda\alpha$. By Lemma 2.1 in \cite{Liux}, we have the following pointwise identity:
	\begin{eqnarray}\label{pt_est}
	\begin{array}{rl}
		&\displaystyle \frac{1}{2} e^{2\lambda\alpha} \Big[p_t+\sum_{i,j=1}^n \left(b^{ij}p_{x_i}\right)_{x_j}\Big]^2 \\
		&\displaystyle\geq \sum_{i,j=1}^n \Big[
		b^{ij}q_{x_i}q_t - \sum_{i'j'=1}^n \Big( 2b^{ij} b^{i'j'} \ell_{x_{i'}}q_{x_i}q_{x_{j'}} -b^{ij} b^{i'j'} \ell_{x_i} q_{x_{i'}} q_{x_{j'}}\Big) - E b^{ij}\ell_{x_i}q^2
		\Big]_{x_j} \\
		& \displaystyle\quad-\frac{1}{2} \Big(\sum_{i,j=1}^n b^{ij}q_{x_i}q_{x_j} - Eq^2\Big)_t + \sum_{i,j=1}^n c^{ij} q_{x_i} q_{x_j}
		-2 \sum_{i,j=1}^n \sum_{i',j'=1}^n \left(b^{ij}q_{x_i}\right)_{x_j} b^{i'j'} \ell_{x_{i'}x_{j'}}q
		+ F q^2,
	\end{array}
	\end{eqnarray}
	where
	\begin{eqnarray*}
	\left\{\!\!\begin{array}{rl}
		&\displaystyle\Psi  = -\sum_{i,j=1}^n \left(b^{ij}_{x_j}\ell_{x_i} + 2b^{ij} \ell_{x_ix_j}\right),\q	E = -\ell_t + \sum_{i, j=1}^n \left(b^{ij}\ell_{x_i}\ell_{x_j}-b^{ij}_{x_j}\ell_{x_i}-b^{ij}\ell_{x_ix_j}\right) - \Psi,\\
		&\displaystyle
		F  = - \frac{1}{2}E_t + \sum_{i, j=1}^n \left[\left(Eb^{ij}\ell_{x_i}\right)_{x_j} - \frac{1}{2} \left(b^{ij}_{x_j}\ell_{x_i}\right)^2 - 2 E b^{ij} \ell_{x_ix_j}\right],\\
		&\displaystyle
		c^{ij} = \frac{1}{2} b^{ij}_t + \sum_{i',j'=1}^{n} \left[ 2b^{ij'} \left(b^{i'j}\ell_{x_{i'}}\right)_{x_{j'}} - \left(b^{ij}b^{i'j'}\ell_{x_{i'}}\right)_{x_{j'}}
		\right].
		\end{array}\right.
	\end{eqnarray*}
	
	Using an argument similar to that in the proof of Proposition 2.1 in \cite{Liux}, and taking into account the differences arising from the choice of $\gamma_0$ discussed at the beginning of the proof, we obtain the following lower bound for $F$:
	\begin{eqnarray}\label{1}
	\begin{array}{rl}
	&\displaystyle F\geq \rho_0^2 |\nabla \psi|^4 \lambda^3\mu^4\varphi^3+
	\rho_0|\nabla \psi|^2  \lambda^2\mu^2  \varphi [\mu e^{12\mu}-e^{\mu(\psi+6)}](-\gamma'_0)\chi_{[0, T/4]}\\[4mm]
	&\displaystyle \quad\quad-C_0\lambda^3\mu^3 \varphi^3
	-C_0\lambda^2\mu  \varphi [\mu e^{12\mu}-e^{\mu(\psi+6)}](-\gamma'_0)\chi_{[0, T/4]},
	\end{array}
	\end{eqnarray}
	where $\chi_{[0, T/4]}$ denotes the characteristic function on the
	interval  $[0, T/4]$. Furthermore,
	\begin{eqnarray}\label{2}
	\begin{array}{rl}
		&\displaystyle \int_Q\left[
		 \sum_{i,j=1}^n c^{ij} q_{x_i} q_{x_j}
		-2 \sum_{i,j=1}^n \sum_{i',j'=1}^n \left(b^{ij}q_{x_i}\right)_{x_j} b^{i'j'} \ell_{x_{i'}x_{j'}}q \right]dxdt\\[6mm]
		&\displaystyle
		 \geq \int_Q \rho_0^2 |\nabla\psi|^2 \lambda\mu^2\varphi |\nabla q|^2 dxdt-C_0\int_Q \lambda\mu\varphi|\nabla q|^2 dxdt
		 -C_0\int_Q \lambda\mu^3\varphi |q||\nabla q| dxdt.
		 \end{array}
	\end{eqnarray}

	Next, using the property \eqref{psi} and the boundary condition $q=0$ on $\Sigma$, the spatial boundary terms give a non-negative contribution as
	\begin{align*}
		\int_\Sigma \sum_{i, j=1}^n \sum_{i', j'=1}^n \left(
		-2 b^{i j} b^{i' j'} \ell_{x_{i'}} q_{x_i}q_{x_{j'}} + b^{i j} b^{i' j'} \ell_{x_i} q_{x_{i'}} q_{x_{j'}} \right) \nu_{j}dSdt\geq 0,
	\end{align*}
	where $\nu_j$ is the $j$-th component of the outward unit normal vector to $\Gamma$.
	Consequently,
	\begin{align}\label{3}
		\int_Q \sum_{i, j=1}^n \left[
		b^{ij}q_{x_i}q_t - \sum_{i', j'=1}^n \left( 2b^{i j} b^{i' j'}
		\ell_{x_{i'}} q_{x_i}q_{x_{j'}} -b^{i j} b^{i' j'} \ell_{x_i} q_{x_{i'}} q_{x_{j'}} - E b^{ij} \ell_{x_i}q^2\right)
		\right]_{x_j}dxdt \geq 0.
	\end{align}
	
	Note that $q(T,x)=0$ for $x\in \Omega$, and therefore the divergence term with respect to the time variable $t$ can be estimated as
	\begin{eqnarray*}
	&&-\frac{1}{2}\int_Q \left[\sum\limits_{i, j=1}^{n} b^{i j}q_{x_i}q_{x_j} dxdt + E q^2 dxdt\right]_t\\
	&&=\frac{1}{2} \int_\Omega \sum\limits_{i, j=1}^{n} b^{i j}(x, 0)q_{x_i}(x, 0)q_{x_j}(x, 0) dx
	-\frac{1}{2} \int_\Omega E(x, 0)q^2(x, 0)dx\\
	&&\geq \frac{1}{2} \int_\Omega \rho_0 |\nabla q(x, 0)|^2 dx-\frac{1}{2}\int_\Omega E(x, 0)q^2(x, 0)dx.
	\end{eqnarray*}
	At $t=0$, we have
	\begin{eqnarray*}
	&&-E(\cdot, 0)\geq
	-\frac{4}{T}\lambda[e^{\mu(\psi+6)}-\mu e^{12\mu}]\cdot \lambda\mu^2 e^{2\mu}	
	-C_0\lambda^2\mu^2 e^{2\mu(\psi+6)}\geq C_0\lambda^2\mu^3 e^{14\mu},
	\end{eqnarray*}
	for sufficiently large $\mu_0$. Hence, the divergence term with respect to $t$ satisfies
	\begin{equation}\label{tt}
	\int_Q  \left[-\frac{1}{2}\sum\limits_{i, j=1}^{n} b^{i j}q_{x_i}q_{x_j} dxdt
	+ E q^2 dxdt\right]_t\geq  C_0\int_\Omega \lambda^2\mu^3 e^{14\mu} q^2(x, 0)dx.	
	\end{equation}

	Integrating the pointwise estimate \eqref{pt_est} over $Q$ and employing estimates \eqref{1}--\eqref{tt}, we obtain
	\begin{eqnarray*}
		&&\int_Q |\nabla \psi|^4 \lambda^3\mu^4\varphi^3 q^2 dxdt
		+ \int_Q  |\nabla\psi|^2 \lambda\mu^2\varphi |\nabla q|^2 dxdt\\
		&&+\int_\Omega \lambda^2 \mu^3 e^{14\mu} q^2(x, 0)dx
		+\int^{T/4}_0\!\!\int_\Omega |\nabla \psi|^2  \lambda^2\mu^2  \varphi \Big[\mu e^{12\mu}-e^{\mu(\psi+6)}\Big](-\gamma'_0)q^2 dxdt\\
		&&\leq C_0 \int_Q \lambda^3\mu^3 \varphi^3 q^2 dxdt+
		C_0\int_Q \lambda\mu\varphi |\nabla q|^2 dxdt\\
		&&\quad+ C_0\int^{T/4}_0\!\!\int_\Omega
		  \lambda^2\mu  \varphi \Big(\mu e^{12\mu}-e^{\mu(\psi+6)}\Big)(-\gamma'_0)q^2 dxdt+C_0\int_Q
		 \theta^2 F^2 dxdt.
	\end{eqnarray*}
	We use the property \eqref{psi} concerning the gradient to conclude the following
 	 \begin{eqnarray*}
		&&\int_Q \lambda^3\mu^4\varphi^3 q^2 dxdt
		+ \int_Q   \lambda\mu^2\varphi |\nabla q|^2 dxdt
		+\int_\Omega \lambda^2 \mu^3 e^{14\mu} q^2(x, 0)dx\\
		&&\leq C_0\int^T_0\int_{\tilde{\omega}_0}  \lambda^3\mu^4 \varphi^3 q^2 dxdt+
		C_0\int^T_0\int_{\tilde{\omega}_0}  \lambda\mu^2\varphi |\nabla q|^2 dxdt\\
		&&\quad+ C_0\int^{T/4}_0\!\!\int_{\tilde{\omega}_0}
		 \lambda^2\mu  \varphi \Big(\mu e^{12\mu}-e^{\mu(\psi+6)}\Big)(-\gamma'_0)q^2 dxdt
		 +C_0\int_Q
		 \theta^2 F^2 dxdt,
	\end{eqnarray*}
	for  sufficiently large $\mu_0$.
	Let us now note that
	$$
	\displaystyle
	\lambda^2\mu  \varphi \cdot \mu e^{12\mu} (-\gamma'_0)
	\leq C_0\lambda^3\mu^4\varphi e^{14\mu}\leq C_0\lambda^3\mu^4 \varphi^3 e^{2\mu} \quad \text{ on } (0,T/4)\times \Omega.
	$$
	Using this in the previous estimate gives
	\begin{eqnarray*}
		&&\int_Q \lambda^3\mu^4\varphi^3 q^2 dxdt
		+ \int_Q   \lambda\mu^2\varphi |\nabla q|^2 dxdt
		+\int_\Omega \lambda^2 \mu^3 e^{14\mu} q^2(x, 0)dx\\
		&&\leq  C_0\int^{T/4}_0\!\!\int_{\tilde{\omega}_0}
		\lambda^3\mu^4 \varphi^3 e^{2\mu} q^2 dxdt+
		C_0\int^T_{T/4}\int_{\tilde{\omega}_0}  \lambda^3\mu^4 \varphi^3 q^2 dxdt\\
		&&\quad+
		C_0\int^T_0\int_{\tilde{\omega}_0}  \lambda\mu^2\varphi |\nabla q|^2 dxdt
		 +C_0\int_Q
		 \theta^2 F^2 dxdt.
	\end{eqnarray*}
	
	Substituting  $q=e^{\lambda\alpha}p$ into this inequality yields	
		\begin{eqnarray}\label{Car1}
		\begin{array}{rl}
		&\displaystyle\int_Q \theta^2\lambda^3\mu^4\varphi^3 p^2 dxdt
		+ \int_Q  \theta^2 \lambda\mu^2\varphi |\nabla p|^2 dxdt
		+\int_\Omega \lambda^2 \mu^3 e^{14\mu} \theta^2(x, 0) p^2(x, 0)dx\\
		&\displaystyle\leq  C_0\int^{T/4}_0\!\!\int_{\tilde{\omega}_0}
		\theta^2\lambda^3\mu^4 \varphi^3 e^{2\mu} p^2 dxdt+
		C_0\int^T_{T/4}\int_{\tilde{\omega}_0} \theta^2 \lambda^3\mu^4 \varphi^3 p^2 dxdt\\
		&\displaystyle\quad+
		C_0\int^T_0\int_{\tilde{\omega}_0}  \theta^2\lambda\mu^2\varphi |\nabla p|^2 dxdt
		 +C_0\int_Q
		 \theta^2 F^2 dxdt.
		 \end{array}
	\end{eqnarray}
	To remove the local gradient term on the right-hand side,  we introduce a smooth cutoff function
	 $\xi\in C_0^\infty(\omega_0)$ such that $\xi\geq 0$ on $\omega_0$ and $\xi = 1$ on $\tilde{\omega}_0$.
	 Multiplying  \eqref{e1} by $\xi \varphi \theta^2 p$ and integrating over $Q$, standard integration by parts gives
	 \begin{align*}
	\displaystyle\int^T_0\int_{\tilde{\omega}_0}  \theta^2\lambda\mu^2\varphi |\nabla p|^2 dxdt
	 \!\leq\! \int_Q  \theta^2\lambda\mu^2\varphi \xi|\nabla p|^2 dxdt  \!\leq\! C_0\int_0^T\int_{\omega_0} \xi \theta^2 \lambda^3\mu^4 \varphi^3 p^2 dxdt
	\!+\! C_0\int_Q \theta^2 F^2 dxdt.
	\end{align*}
Combining this with  (\ref{Car1})  yields
\begin{eqnarray}\label{Car1**}
		\begin{array}{rl}
		&\displaystyle\int_Q \theta^2\lambda^3\mu^4\varphi^3 p^2 dxdt
		+ \int_Q  \theta^2 \lambda\mu^2\varphi |\nabla p|^2 dxdt
		+\int_\Omega \lambda^2 \mu^3 e^{14\mu} \theta^2(x, 0) p^2(x, 0)dx\\
		&\displaystyle\leq  C_0\int^{T/4}_0\!\!\int_{\omega_0}
\theta^2\lambda^3\mu^4 \varphi^3 e^{2\mu} p^2 dxdt+
		C_0\int^T_{T/4}\int_{\omega_0} \theta^2 \lambda^3\mu^4 \varphi^3 p^2 dxdt+
		C_0\int_Q
		 \theta^2 F^2 dxdt.
		 \end{array}
	\end{eqnarray}

	Finally,  to bound   the   $p_t$ and $p_{x_i x_j}$ terms,
	we use the identity $a^2+b^2=(a+b)^2-2ab$:
	\begin{eqnarray}\label{car_id!}
	\begin{array}{rl}
		&\displaystyle\int_Q \theta^2 \lambda^{-1}\varphi^{-1} \left[
		p_t^2+\big|\sum\limits_{i, j=1}^{n} (b^{i j}p_{x_i})_{x_j}\big|^2\right]dxdt \\
	&\displaystyle= \int_Q \theta^2 \lambda^{-1}\varphi^{-1}  F^2dxdt-
		2\int_Q \theta^2 \lambda^{-1}\varphi^{-1} p_t\cdot \sum\limits_{i, j=1}^{n} (b^{i j}p_{x_i})_{x_j} dxdt.
		\end{array}
	\end{eqnarray}
Integration by parts then yields	
\begin{eqnarray*}
&&-2\int_Q \theta^2 \lambda^{-1}\varphi^{-1} p_t\cdot \sum\limits_{i, j=1}^{n} (b^{i j}p_{x_i})_{x_j} dxdt
=2\int_Q \sum\limits_{i, j=1}^{n}
 (\theta^2 \lambda^{-1}\varphi^{-1} p_t)_{x_j}  b^{i j}p_{x_i} dxdt\\
 &&=\int_Q \theta^2 \lambda^{-1}\varphi^{-1} \sum\limits_{i, j=1}^{n} (b^{i j}p_{x_i}p_{x_j})_tdxdt
 -\int_Q \theta^2 \lambda^{-1}\varphi^{-1} \sum\limits_{i, j=1}^{n} b^{i j}_t p_{x_i}p_{x_j} dxdt\\
 &&\quad+2\int_Q \sum\limits_{i, j=1}^{n}  b^{i j}(\theta^2 \lambda^{-1}\varphi^{-1})_{x_j} p_{x_i} p_t dxdt\\
 &&\leq -\int_Q (\theta^2 \lambda^{-1}\varphi^{-1})_t
  \sum\limits_{i, j=1}^{n} b^{i j}p_{x_i}p_{x_j} dxdt
 -\int_Q \theta^2 \lambda^{-1}\varphi^{-1} \sum\limits_{i, j=1}^{n} b^{i j}_t p_{x_i}p_{x_j} dxdt\\
&&\quad +2\int_Q \sum\limits_{i, j=1}^{n}  b^{i j}(\theta^2 \lambda^{-1}\varphi^{-1})_{x_j} p_{x_i} p_t dxdt\\
&&\leq \frac{1}{2} \int_Q \theta^2 \lambda^{-1}\varphi^{-1} p_t^2 dxdt
+C_0\int_Q \theta^2 \lambda\mu^2\varphi |\nabla p|^2dxdt.
\end{eqnarray*}
Inserting this estimate  into (\ref{car_id!}) gives
$$
\displaystyle \int_Q \theta^2 \lambda^{-1}\varphi^{-1} \left[
		p_t^2+\big|\sum\limits_{i, j=1}^{n} (b^{i j}p_{x_i})_{x_j}\big|^2\right]dxdt
		\leq C_0\int_Q \theta^2 F^2 dxdt
+C_0\int_Q \theta^2 \lambda\mu^2\varphi |\nabla p|^2dxdt.
$$
This, together with (\ref{Car1**}),  yields   	
 the estimate
 \begin{eqnarray*}
\begin{array}{rl}
	&\displaystyle \int_\Omega
	\lambda^2\mu^3 e^{14\mu}
	\theta^2(x, 0) p^2(x, 0) dx
	+\int_Q \theta^2\left(
	\lambda^3\mu^4\varphi^3 p^2+\lambda\mu^2\varphi |\nabla p|^2\right) dxdt\\
	&\quad\displaystyle +\int_Q \theta^2 \lambda^{-1}\varphi^{-1}\left(
	|p_t|^2+ \big|\sum\limits_{i, j=1}^{n} (b^{i j}p_{x_i})_{x_j}\big|^2\right) dxdt\\
	&\displaystyle
	 \leq C_0
	\left(\int^{T/4}_{0}\!\!\!\int_{\omega_0}\theta^2\lambda^3\mu^4 e^{2\mu}\varphi^3 p^2dxdt+
	\int^{T}_{T/4}\!\int_{\omega_0}
	\theta^{2} \lambda^3\mu^4\varphi^{3}  p^2dxdt
	+\int_Q\theta^2 F^2dxdt\right).
\end{array}
\end{eqnarray*}

On the other hand,   the elliptic regularity estimates further  imply
\begin{eqnarray*}
&&\sum\limits_{i, j=1}^n \int_Q \theta^2\lambda^{-1}\varphi^{-1} |p_{x_i x_j}|^2dxdt \\
&&\leq C_0\sum\limits_{i, j=1}^n \int^T_0 |\theta \lambda^{-1/2}\varphi^{-1/2} p|^2_{H^2(\Omega)}dt+
C_0\int_Q \theta^2\Big(\lambda^3\mu^4\varphi^3 p^2
+\lambda\mu^2\varphi |\nabla p|^2\Big)dxdt\\
&&\leq C_0\int_Q  \theta^2 \lambda^{-1}\varphi^{-1} \big|\sum\limits_{i, j=1}^n
 (b^{i j} p_{x_i})_{x_j}\big|^2dxdt+
C_0\int_Q \theta^2\Big(\lambda^3\mu^4\varphi^3 p^2
+\lambda\mu^2\varphi |\nabla p|^2\Big)dxdt\\
&&
	 \leq C_0
	\left(\int^{T/4}_{0}\!\!\!\int_{\omega_0}\theta^2\lambda^3\mu^4 e^{2\mu}\varphi^3 p^2dxdt+
	\int^{T}_{T/4}\!\int_{\omega_0}
	\theta^{2} \lambda^3\mu^4\varphi^{3}  p^2dxdt
	+\int_Q\theta^2 F^2dxdt\right).
\end{eqnarray*}
This completes the proof of Proposition \ref{lemma1}.
\end{proof}

\begin{remark}
	Let us highlight some important features of the Carleman estimate \eqref{carleman-1} $($in Proposition $\ref{lemma1})$ in comparison with some existing results.
	\begin{enumerate}
		\item [$(a)$] The Carleman estimate \eqref{carleman-1} explicitly provides a weighted $L^2$-estimate for the initial value $p(\cdot,0)$ on the left-hand side. Similar estimates have been obtained for stochastic parabolic equations in $\cite{zxl}$, where stronger regularity assumptions were imposed on the principal coefficients $b^{ij}$ than the $C^{1,1}(\overline Q)$ regularity assumed here.
		
		\item [$(b)$] A Carleman estimate for linear parabolic operators with $C^{1,1}$ principal coefficients has also been established in the article $\cite{fu}$. However, the weight function used there is singular at both $t=0$ and $t=T$, and therefore does not directly provide a weighted $L^2$-estimate for the initial value.
		By combining that estimate with suitable energy estimates, one can obtain a bound for the initial value in terms of the non-homogeneous term. Our estimate, in contrast, directly controls the initial value by the weighted non-homogeneous term.
	\end{enumerate}	
\end{remark}

As a direct corollary of Proposition \ref{lemma1},
fixing  $\mu$ as a positive constant yields the following corollary.
  \begin{corollary}\label{corollary1*}
There exist  positive constants $\mu_0$,   $\lambda_0$ and $C_1$
such that for any $\mu\geq \mu_0$ and
$\lambda\geq \lambda_0$,
any solution $p$ to $(\ref{e1})$  satisfies
\begin{eqnarray}\label{carleman-11}
\begin{array}{ll}
	&\displaystyle \int_Q \theta^2 \left[ \lambda^{-1}\gamma_0^{-1}\left(
	|p_t|^2+ \sum\limits_{i, j=1}^{n} \big|p_{x_i x_j}\big|^2\right)
	+ \lambda\gamma_0 |\nabla p|^2 + \lambda^3\gamma_0^3 |p|^2 \right] dxdt\\[6mm]
	&\displaystyle\quad+ \int_\Omega \lambda^2 \theta^2(x, 0) p^2(x, 0) dx\\[2mm]
	& \displaystyle\leq C_1
	\left(\int^T_0\int_{\omega_0} \theta^2\lambda^3\gamma_0^3 p^2dxdt
	+\int_Q\theta^2 F^2dxdt\right),  \quad \forall\  p_T\in L^2(\Omega),
	\end{array}
\end{eqnarray}
where the constant $C_1$ depends  on $\rho_0$,  $\Omega$,  $\omega_0$, $n$, $T$, $\lambda_0$, $\mu_0$, $\Lambda$ and $\mu$, but is independent of $\lambda$.
\end{corollary}

Similar to the arguments in \cite{IY},  Corollary \ref{corollary1*} yields
 the following weight-indexed  Carleman estimate.
\begin{corollary}\label{corollary2*}
	For any $m\in\mathbb{R}$, there exist constants positive constants $\mu_{0,m}$, $\lambda_{0,m}$ and $\hat{C}_m$ such that for any $\mu\geq \mu_{0,m}$ and $\lambda\geq \lambda_{0,m}$, the solution $p$ to $(\ref{e1})$  satisfies
	\begin{eqnarray}\label{carleman-111}
		\begin{array}{rl}
			&\displaystyle \int_\Omega
			\lambda^{2+m} \gamma_0^m
			\theta^2(x, 0) p^2(x, 0) dx
			+\int_Q \theta^2\left( \lambda^{3+m}\gamma_0^{3+m} |p|^2
			+\lambda^{1+m}\gamma_0^{1+m} \varphi |\nabla p|^2\right) dxdt\\
			&\quad\displaystyle +\int_Q \theta^2 \lambda^{-1+m}\gamma_0^{-1+m}\left(
			|p_t|^2+ \sum\limits_{i, j=1}^{n}\big|p_{x_i x_j}\big|^2\right) dxdt\\
			&\displaystyle
			 \leq \hat{C}_m \left( \int^T_0\int_{\omega_0} \theta^2\lambda^{3+m}\gamma_0^{3+m} |p|^2dxdt
			+\int_Q\theta^2\lambda^m\gamma_0^m |F|^2dxdt \right),  \quad\quad \forall\  p_T\in L^2(\Omega),
		\end{array}
	\end{eqnarray}
where $\hat{C}_m$ depends  on  $\rho_0$, $m$, $\Omega$,  $\omega_0$, $n$, $T$, $\lambda_{0,m}$, $\mu_{0,m}$,  $\Lambda$ and $\mu$, but is independent of $\lambda$.
\end{corollary}

\begin{proof}
Set $P=\lambda^{\frac{m}{2}}\gamma_0^{\frac{m}{2}}p$. Applying
 the Carleman estimate in Corollary \ref{corollary1*} to $P$, we immediately  obtain
  the desired estimate (\ref{carleman-111}).
 \end{proof}

 Next,   consider the following parabolic equation:
\begin{eqnarray}\label{e1**}
	\left\{
	\begin{array}{ll}
		p_t+\sum\limits_{i, j=1}^n \left(b^{i j} p_{x_i}\right)_{x_j}
		=F+\mbox{div} {\bf G}   &\mbox{ in   }Q,\\[2mm]
		p=0  &\mbox{ on  }\Sigma,\\[2mm]
		p(x, T)=p_T(x)  &\mbox{ in  }\Omega,
	\end{array}
	\right.
\end{eqnarray}
where   $p_T\in L^2(\Omega)$,
$F\in L^2(Q)$ and ${\bf G}\in (L^2(Q))^n$.

 Analogous to  Proposition 2.1 in \cite{LX},  we have the following Carleman estimate for (\ref{e1**}).
  \begin{corollary}\label{corollary3*}
There exist  positive constants $\mu_1$,   $\lambda_1$ and $C_2$
such that for any $\mu\geq \mu_1$ and
$\lambda\geq \lambda_1$,
any solution $p$ to $(\ref{e1**})$  satisfies
\begin{eqnarray}\label{carleman3**}
\begin{array}{rl}
	&\displaystyle \int_\Omega
	\lambda^2
	\theta^2(x, 0) p^2(x, 0) dx
	+\int_Q \theta^2\left(
	\lambda^3\gamma_0^3 p^2+\lambda\gamma_0 |\nabla p|^2\right) dxdt\\
	&\displaystyle
	 \leq C_2
	\left(\int^T_0\int_{\omega_0} \theta^2\lambda^3\gamma_0^3 p^2dxdt
	+\int_Q\theta^2 F^2dxdt+\int_Q\theta^2 \lambda^2\gamma_0^2  |{\bf G}|^2dxdt\right),  \quad \forall\  p_T\in L^2(\Omega),
\end{array}
\end{eqnarray}
where  $C_2$ depends  on $\rho_0$,  $\Omega$,  $\omega_0$, $n$, $T$, $\lambda_1$, $\mu_1$, $\Lambda$ and $\mu$, but is independent of $\lambda$.
\end{corollary}

 Similar to the arguments in \cite{IY} again,  Corollary \ref{corollary3*} yields
 the following Carleman estimate.
 \begin{corollary}\label{corollary4*}
	For any  $m\in\mathbb{R}$, there exist constants positive constants $\mu_{1,m}$,
	  $\lambda_{1,m}$ and $\tilde{C}_m$ such that
	  for any $\mu\geq \mu_{1,m}$ and $\lambda\geq \lambda_{1,m}$,
	the solution $p$ to $(\ref{e1**})$  satisfies
	\begin{eqnarray}\label{carleman-111**}
\begin{array}{rl}
	&\displaystyle \int_\Omega
	\lambda^{2+m} \gamma_0^m
	\theta^2(x, 0) p^2(x, 0) dx
	+\int_Q \theta^2\left(
	\lambda^{3+m}\gamma_0^{3+m} p^2+\lambda^{1+m}\gamma_0^{1+m} |\nabla p|^2\right) dxdt\\
	&\displaystyle
	 \leq \tilde{C}_m
	\left(\int^T_0\int_{\omega_0} \theta^2\lambda^{3+m}\gamma_0^{3+m} p^2dxdt
	+\int_Q\theta^2\lambda^m\gamma_0^m F^2dxdt\right.\\
	&\left.\displaystyle \quad\quad\quad\quad+\int_Q\theta^2 \lambda^{2+m}\gamma_0^{2+m}  |{\bf G}|^2dxdt\right),  \quad\quad\forall\  p_T\in L^2(\Omega),
\end{array}
\end{eqnarray}
where $\tilde{C}_m$ depends on $\rho_0$,  $m$, $\Omega$,  $\omega_0$, $n$, $T$, $\lambda_{1,m}$, $\mu_{1,m}$,  $\Lambda$ and $\mu$, but is independent of $\lambda$.
\end{corollary}

\section{The direct method for  controllability approximation}\label{chapter3}
This section is devoted to proving Theorem \ref{th2} via a direct approach  developed in \cite{LX}. More precisely, we establish the null controllability of both the quasi-linear parabolic system \eqref{aeq1} (for sufficiently small initial data) and the linear parabolic system \eqref{aeq1*}. More importantly, we provide a quantitative estimate for the difference between their associated solutions. In what follows, the parameter $\mu$ appearing in the Carleman estimates of \Cref{chapter2} is fixed.

First, we establish the null controllability of the linearized system associated with \eqref{aeq1} using control functions in $C^{1+\delta,\frac{1+\delta}{2}}(\overline Q)$ for $\delta\in(0,1)$. To this end, we define the set
$$
K=\Big\{
z\in C^{3+\delta,\frac{3+\delta}{2}}(\overline Q)
\ \Big|\
|z|_{C^{3+\delta,\frac{3+\delta}{2}}(\overline Q)}\leq1
\Big\},
$$
and, for each $z\in K$, we consider the following linearized system associated with \eqref{aeq1}:
\begin{eqnarray}\label{old1}
	\left\{
	\begin{array}{ll}
		y_t-\sum\limits_{i, j=1}^n \left(a^{ij}(z, \nabla z)y_{x_i}\right)_{x_j} + \tilde{a}(z)y+\tilde{B}(z)\cdot\nabla y=\xi_0 u &\mbox{ in }Q, \\ \ns
		y=0 &\mbox{ on }\Sigma, \\ \ns
		y(x, 0)=y_0(x) &\mbox{ in }\Omega,
	\end{array}
	\right.
\end{eqnarray}
where
\begin{equation}\label{***?}
\displaystyle \tilde{a}(z)=\int^1_0 f_s(\tau z, \tau\nabla z)d\tau\quad \mbox{ and }\quad
\tilde{B}(z)=\int^1_0 \nabla_\zeta f(\tau z, \tau\nabla z)d\tau.
\end{equation}

By the definitions of H\"older spaces,   for any $g\in C^2(\mathbb{R}^{1+n})$ and $z\in
K$, we have
$g(z, \nabla z)\in C^{1+\delta, \frac{1+\delta}{2}}(\overline{Q})$. Moreover,
$$
|g(z, \nabla z)|_{C^{1+\delta, \frac{1+\delta}{2}}(\overline{Q})}
\leq C\left(\!|g|_{C(B_1)}+
|g|_{C^1(B_1)} |z|_{C^{3+\delta, \frac{3+\delta}{2}}(\overline{Q})}
+|g|_{C^2(B_1)} |z|^2_{C^{3+\delta, \frac{3+\delta}{2}}(\overline{Q})}\!\right)\leq C|g|_{C^2(B_1)},
$$
where $B_1=\big\{\ (s, \zeta)\in\mathbb R\times\mathbb R^n\ \big|\
|s|\leq 1, |\zeta|\leq 1\ \big\}$.
Hence, $a^{i j}(z, \nabla z),   a^{i j}_s(z, \nabla z),  \ \tilde{a}(z)\in C^{1+\delta, \frac{1+\delta}{2}}(\overline{Q})$,
$ \nabla_\zeta a^{i j}(z, \nabla z), \ \tilde{B}(z)\in (C^{1+\delta,
 \frac{1+\delta}{2}}(\overline{Q}))^n$ and $a^{i j}(z, \nabla z)\in C^{1, 1}(\overline{Q})
$ for any $z\in K$. Furthermore,
\begin{eqnarray}\label{old2}
\begin{array}{rl}
&\displaystyle\sum\limits_{i, j=1}^{n}\left[
|a^{i j}(z, \nabla z)|_{C^{1+\delta, \frac{1+\delta}{2}}(\overline{Q})}
+|\left(a^{i j}(z, \nabla z)\right)_{x_j}|_{C^{1+\delta, \frac{1+\delta}{2}}(\overline{Q})}\right]\\
&\quad+|\tilde{a}(z)|_{C^{1+\delta, \frac{1+\delta}{2}}(\overline{Q})}+|\tilde{B}(z)|_{\big(C^{1+\delta, \frac{1+\delta}{2}}(\overline{Q})\big)^n}+|a^{i j}(z, \nabla z)|_{C^{1, 1}(\overline{Q})}\\
&\leq C\left(|f|_{C^3(B_1)}+\sum\limits_{i, j=1}^{n} |a^{i j}|_{C^3(B_1)}
\right)\leq C.\end{array}\end{eqnarray}
Here,   $C$ denotes a generic    positive constant, independent of $z$,
which depends  on
$\rho_0, n, T, \Omega, \omega, \omega_0$, $\delta$, $f$ and  $a^{i j}$,
and may vary  from line to line.
By the local well-posedness  for linear parabolic equations, for any
$y_0\in C^{3+\delta}(\overline\Omega)$
and $u\in C^{1+\delta, \frac{1+\delta}{2}}(\overline{Q})$ satisfying
the  compatibility conditions, (\ref{old1})
admits a unique solution $y\in C^{3+\delta, \frac{3+\delta}{2}}(\overline{Q})$, with
$
|y|_{C^{3+\delta, \frac{3+\delta}{2}}(\overline{Q})}\leq C\big(|y_0|_{C^{3+\delta}(\Omega)}+
|u|_{C^{1+\delta, \frac{1+\delta}{2}}(\overline{Q})}\big).
$

The main result of this section establishes the null controllability of the linearized system \eqref{old1},  and provides a quantitative relationship between its null controls and those of the linear system \eqref{aeq1*}. More precisely, we have the following result.
\begin{theorem}\label{t3}
Assume that ${\bf (H)}$ holds.
Then for any $y_0\in C^{3+\delta}(\overline\Omega)$ satisfying the compatibility conditions, there exist two controls  $u^*,  v^*\in
 C^{1+\delta, \frac{1+\delta}{2}}(\overline{Q})$ with
 $$
 	|u^*|_{C^{1+\delta, \frac{1+\delta}{2}}(\overline{Q})}
 	+|v^*|_{C^{1+\delta, \frac{1+\delta}{2}}(\overline{Q})}\leq
 	C
 	|y_0|_{C^{2}(\overline\Omega)},$$
 such that the associated solutions $y^*$
 and $w^*$ to $(\ref{old1})$  and $(\ref{aeq1*})$ satisfy
 $$
 	 y^*(\cdot, T; u^*) = w^*(\cdot, T; v^*) = 0 \text{ in }\Omega.$$
Moreover, the difference between the controls $u^*$ and $v^*$ can be estimated as
\begin{equation}
	|u^*-v^*|_{C^{1+\delta, \frac{1+\delta}{2}}(\overline{Q})}
	\leq C |z|^{1/2}_{C^{2, 0}(\overline{Q})} |y_0|_{C^{2}(\overline\Omega)}.
\end{equation}
\end{theorem}

To prove \Cref{t3}, we first require the following approximate null controllability results for the linearized system \eqref{old1} and the linear system \eqref{aeq1*}.
\begin{lemma}\label{lemma2}
For any $y_0\in C^{3+\delta}(\overline\Omega)$ satisfying the  compatibility conditions and
$\varepsilon>0$,  there  exist  controls
$u_\varepsilon,  v_\varepsilon\in L^2(Q)$
such that the corresponding solutions
$y_\varepsilon$ and $w_\varepsilon$  to $(\ref{old1})$ and $(\ref{aeq1*})$,
 respectively,  satisfy
\begin{eqnarray}\label{old*}
\begin{array}{ll}
&\displaystyle\int_Q  \theta^{-2}\lambda^{-3}\gamma_0^{-3} \left(u_{\varepsilon}^2+v_{\varepsilon}^2\right) dxdt
+\int_Q   \theta^{-2} \left(y_{\varepsilon}^2+w_{\varepsilon}^2\right) dxdt+\displaystyle\frac{1}{\varepsilon}\int_\Omega \left[y_{\varepsilon}^2(x, T)+w_{\varepsilon}^2(x, T)\right]dx\\[4mm]
&\leq
 Ce^{C\lambda} |y_0|^2_{L^2(\Omega)},
 \end{array}
\end{eqnarray}
for sufficiently large $\lambda$, where $C$ is a positive constant depending on $\rho_0$, $\Omega$, $\omega_0$, $n$, $T$, $\lambda_0$, $\mu_0$, $\Lambda$ and $\mu$, but is independent of $\lambda$.
\end{lemma}

\begin{proof}
 For any $\varepsilon>0$ and $\lambda\geq\lambda_0$, consider the  optimal control problem associated with \eqref{old1}:
 $$
{\bf (P_\varepsilon)} \quad  \quad
\min\limits_{u\in  \mathcal V} \left\{\
 \frac{1}{2}\int_Q  \theta^{-2}  y^2 dxdt+\frac{1}{2}\int_Q  \theta^{-2}\lambda^{-3}\gamma_0^{-3} u^2dxdt
 +\frac{1}{2\varepsilon} \int_\Omega  y^2(x, T)dx\  \right\},
 $$
 where
 $\displaystyle
 \mathcal V=\Big\{\   u\in L^2(Q)\  \Big|\  \int_Q  \theta^{-2}\gamma_0^{-3} u^2dxdt<+\infty\  \Big\}$
 and $y$ is the  solution to (\ref{old1}) corresponding to the control $u\in \mc{V}$.
 By standard arguments, ${\bf(P_\varepsilon)}$ admits a unique optimal solution
$(u_\varepsilon, y_\varepsilon)\in \mathcal{V}\times L^2(0, T; H^1_0(\Omega))$. Furthermore,
the optimal control $u_\varepsilon$ is characterized by
$
u_\varepsilon=\theta^2\lambda^3\gamma_0^3\xi_0 p_\varepsilon$
 in  $ Q$,
where $p_\varepsilon$ solves the following adjoint system:
\begin{eqnarray}\label{54}
\left\{
\begin{array}{ll}
p_{\varepsilon,  t}+\sum\limits_{i, j=1}^n \left(a^{ij}(z, \nabla z)p_{\varepsilon, x_i}\right)_{x_j}
-\tilde{a}(z)p_\varepsilon+\mbox{div}\left(\tilde{B}(z)p_\varepsilon\right) =
\theta^{-2}  y_\varepsilon  &\mbox{ in   }Q,\\[2mm]
p_\varepsilon=0  &\mbox{ on  }\Sigma,\\[2mm]
p_\varepsilon(x, T)=-\displaystyle\frac{1}{\varepsilon}y_\varepsilon(x, T)
&\mbox{ in  }\Omega.
\end{array}
\right.
\end{eqnarray}

Applying  Corollary \ref{corollary1*}  with
$F=\tilde{a}(z)p_\varepsilon+\theta^{-2}  y_\varepsilon-\mbox{div}(\tilde{B}(z)p_\varepsilon)$ and using \eqref{old2}, we obtain
\begin{align*}
& \int_Q \theta^2\Big(\lambda\gamma_0|\nabla p_\varepsilon|^2+
\lambda^{3}\gamma_0^{3} p_\varepsilon^2\Big)dxdt
+\int_\Omega \lambda^2\theta^2(x, 0)p_\varepsilon^2(x, 0)dx\\
&+\int_Q \theta^2 \lambda^{-1}\gamma_0^{-1}\Big(
	|p_{\varepsilon, t}|^2+ \sum\limits_{i, j=1}^{n} \big|p_{\varepsilon, x_i x_j}\big|^2\Big) dxdt\\
&\leq C
\left[
\int^T_0\int_{\omega_0}\theta^{2}\lambda^{3} \gamma_0^{3}  p_\varepsilon^2dxdt
\!+\!\int_Q \theta^{-2} y^2_\varepsilon dxdt\!+\!
\int_Q \theta^2
(p^2_\varepsilon+|\nabla p_\varepsilon|^2) dxdt\right],
\end{align*}
for every $\lambda\geq\lambda_0$.
 Throughout this proof, $C$ denotes a generic positive constant that may vary from line to line. Moreover,  the constant $C$ is independent of $\lambda$,  and depends only on $\rho_0$, $\Omega$, $\omega_0$, $n$, $T$, $\lambda_0$, $\mu_0$, $\Lambda$ and $\mu$.
Choosing $\lambda$ sufficiently large, the last two terms on the right-hand side can be absorbed into the left-hand side, yielding
\begin{eqnarray}\label{old3}
\begin{array}{rl}
&\displaystyle \int_Q \theta^2\left(\lambda\gamma_0|\nabla p_\varepsilon|^2+
\lambda^{3}\gamma_0^{3} p_\varepsilon^2\right)dxdt
+\int_\Omega \lambda^2\theta^2(x, 0)p_\varepsilon^2(x, 0)dx\\
&\displaystyle\quad+\int_Q \theta^2 \lambda^{-1}\gamma_0^{-1}\left(
|p_{\varepsilon, t}|^2+ \sum\limits_{i, j=1}^{n} \big|p_{\varepsilon, x_i x_j}\big|^2\right) dxdt\\
&\displaystyle\leq C
\left(
\int^T_0\int_{\omega_0}\theta^{2}\lambda^{3} \gamma_0^{3}  p_\varepsilon^2dxdt
\!+\!\int_Q \theta^{-2} y^2_\varepsilon dxdt\right).
\end{array}
\end{eqnarray}

Using the duality relation between \eqref{old1} and \eqref{54},  and noting that
$\theta^2(x, 0)\geq e^{-2\lambda \mu e^{12\mu}}$ in $\Omega$,  together with the estimate \eqref{old3}, we get
\begin{eqnarray*}
&&\displaystyle\frac{1}{\varepsilon}\int_\Omega y_\varepsilon^2(x, T)dx+
\int_Q \xi^2_{0} \theta^2 \lambda^3 \gamma_0^3 p_\varepsilon^2dxdt
+\int_Q \theta^{-2}  y_\varepsilon^2 dxdt=-\int_\Omega y_0(x)p_\varepsilon(x, 0)dx\\
&&
\leq |y_0|_{L^2(\Omega)}|p_\varepsilon(\cdot, 0)|_{L^2(\Omega)}\leq Ce^{C\lambda}|y_0|_{L^2(\Omega)} \left(
\int^T_0\int_{\omega_0}\theta^{2}\lambda^{3} \gamma_0^{3}  p_{\varepsilon}^2dxdt
+\int_Q \theta^{-2}  y^2_{\varepsilon} dxdt\right)^{1/2},
\end{eqnarray*}
 which  implies
\begin{eqnarray}\label{55}
\begin{array}{ll}
&\displaystyle\frac{1}{\varepsilon}\int_\Omega y_\varepsilon^2(x, T)dx+
\int_Q \xi^2_{0} \theta^2 \lambda^3 \gamma_0^3 p_\varepsilon^2dxdt
+\int_Q \theta^{-2} y_\varepsilon^2 dxdt\leq
 Ce^{C\lambda}|y_0|^2_{L^2(\Omega)}.
\end{array}
\end{eqnarray}

Next,  we apply  similar arguments for  the linear controlled system \eqref{aeq1*}. For any $\varepsilon>0$,  define the control
$$
v_\varepsilon=\theta^2\lambda^3\gamma_0^3\xi_0 q_\varepsilon
\quad\quad\quad \mbox{ in  }\   Q,
$$
where $q_\varepsilon$ solves the following adjoint system
\begin{eqnarray}\label{56}
\left\{
\begin{array}{ll}
q_{\varepsilon,  t}+\sum\limits_{i, j=1}^n \left(\tilde{a}^{ij} q_{\varepsilon, x_i}\right)_{x_j}=
\theta^{-2} w_\varepsilon  &\mbox{ in   }Q,\\[2mm]
q_\varepsilon=0  &\mbox{ on  }\Sigma,\\[2mm]
q_\varepsilon(x, T)=-\displaystyle\frac{1}{\varepsilon}w_\varepsilon(x, T)
&\mbox{ in  }\Omega,
\end{array}
\right.
\end{eqnarray}
and $w_\varepsilon$ is the solution to (\ref{aeq1*}) corresponding to $v=v_\varepsilon$.
Proceeding as in the derivations of \eqref{old3} and \eqref{55}, we respectively obtain
\begin{eqnarray}\label{old3*&}
\begin{array}{rl}
&\displaystyle \int_Q \theta^2\left(\lambda\gamma_0|\nabla q_\varepsilon|^2+
\lambda^{3}\gamma_0^{3} q_\varepsilon^2\right)dxdt
+\int_\Omega \lambda^2\theta^2(x, 0)q_\varepsilon^2(x, 0)dx\\
&\displaystyle\q+\int_Q \theta^2 \lambda^{-1}\gamma_0^{-1}\left(
	|q_{\varepsilon, t}|^2+ \sum\limits_{i, j=1}^{n} \big|q_{\varepsilon, x_i x_j}\big|^2\right) dxdt\\
&\displaystyle\leq C
\left(
\int^T_0\int_{\omega_0}\theta^{2}\lambda^{3} \gamma_0^{3}  q_\varepsilon^2dxdt
\!+\!\int_Q \theta^{-2} w^2_\varepsilon dxdt\right),
\end{array}
\end{eqnarray}
and
\begin{eqnarray}\label{55*}
\begin{array}{ll}
&\displaystyle\frac{1}{\varepsilon}\int_\Omega w_\varepsilon^2(x, T)dx+
\int_Q \xi^2_{0} \theta^2 \lambda^3 \gamma_0^3 q_\varepsilon^2dxdt
+\int_Q \theta^{-2}  w_\varepsilon^2 dxdt\leq
 Ce^{C\lambda} |y_0|^2_{L^2(\Omega)}.
\end{array}
\end{eqnarray}
Consequently, combining (\ref{55}) and (\ref{55*}) yields
 the desired estimate (\ref{old*}).
\end{proof}

Using  \Cref{lemma2}, one can obtain the following estimate on the gradient of  $y_\varepsilon$ and $w_\varepsilon$.
\begin{corollary}\label{corollary2}
	For any $y_0\in C^{3+\delta}(\overline\Omega)$ satisfying the
	compatibility conditions and
	$\varepsilon>0$, the solutions $y_\varepsilon$ and $w_\varepsilon$ to the systems \eqref{old1} and \eqref{aeq1*} with controls $u_\varepsilon,  v_\varepsilon\in L^2(Q)$, respectively, satisfy the following estimate
	\begin{align}\label{grad_est}
		\int_Q \gamma_0^{-2} \theta^{-2} \left( \left|\nabla y_\varepsilon\right|^2 + \left|\nabla w_\varepsilon\right|^2\right)dxdt \leq Ce^{C\lambda} \left|y_0\right|^2_{L^2(\Omega)},
	\end{align}
	where $C>0$ is a positive constant, independent of $\lambda$.
\end{corollary}
\begin{proof}
	Let $\{\delta_n\}_{n=1}^\infty$ be a sequence of positive real numbers converging to $0.$
	We define the functions $\gamma_n$, $\alpha_n$ and $\theta_n$ as follows
	\begin{align*}
		\gamma_n(t) := \begin{cases}
			\gamma_0(t) & t\in [0,3T/4]\\
			\frac{1}{T-t+\delta_n} & t\in [3T/4,T]
		\end{cases},
		\quad \alpha_n (x,t) := \left[e^{\mu(\psi(x)+6)}-
		\mu e^{12\mu}\right]\gamma_n(t),
		\quad \theta_n := e^{\lambda\alpha_n}.
	\end{align*}
	Note that $\gamma_n\leq \gamma_0$, which implies  $-\alpha_n\leq -\alpha$ and $\theta_n^{-1}\leq \theta^{-1}$.
	
	Multiplying  the equation governed  by the pair $(u_\varepsilon, y_\varepsilon)$ by $\gamma_0^{-2}\theta_n^{-2}y_\varepsilon$ and performing  integration by parts over $Q$,
	and  utilizing the facts that $y=0$ on $\Sigma$ and$\gamma_0^{-1}(T)=0$, we obtain
		\begin{eqnarray*}
		&&\int_Q \gamma_0^{-2} \sum\limits_{i, j=1}^n a^{ij}(z, \nabla z) y_{\varepsilon,x_i} \left(\theta_n^{-2}y_\varepsilon\right)_{x_j} dxdt\\
		&&= \int_Q \xi_0 \gamma_0^{-2}\theta_n^{-2} u_\varepsilon y_\varepsilon dxdt
		+ \frac{1}{2} \int_Q \left(\gamma_0^{-2}\theta_n^{-2}\right)_t\left|y_\varepsilon\right|^2  dxdt
		+\frac{1}{8} \int_\Omega \theta_n^{-2}(x, 0) \left|y_0(x)\right|^2 dx\\
		&&\quad-  \int_Q \gamma_0^{-2}\theta_n^{-2} \tilde{a}(z) \left|y_\varepsilon
		\right|^2 dxdt
		-  \int_Q \gamma_0^{-2}\theta_n^{-2}y_\varepsilon \tilde{B}(z)\cdot\nabla y_\varepsilon dxdt.
	\end{eqnarray*}
	Note that $\left|\left(\gamma_0^{-2}\theta_n^{-2}\right)_t\right|\leq C \lambda^2 \theta^{-2}$. Using this fact along with the property \eqref{b_ij_cond} in the above estimate, we deduce
	\begin{eqnarray*}
		&&\rho_0 \int_Q  \theta_n^{-2} \gamma_0^{-2}\left| \nabla y_{\varepsilon}\right|^2  dxdt\\
		&&\leq  \int_Q \xi_0 \gamma_0^{-2}\theta^{-2}
		\left|u_\varepsilon\right| \left|y_\varepsilon\right|dxdt
		+ C \lambda^2 \int_Q \theta^{-2}\left|y_\varepsilon\right|^2  dxdt
		+ e^{C\lambda}\left|y_0\right|_{L^2(\Omega)}^2\\
		&&\quad
		+ \frac{1}{2} \int_Q \gamma_0^{-2} \sum\limits_{i, j=1}^n \left(a^{ij}(z, \nabla z) \left(\theta_n^{-2}\right)_{x_j}\right)_{x_i} \left|y_\varepsilon\right|^2 dxdt.
	\end{eqnarray*}
	Observe that $\left|\left(a^{ij}(z, \nabla z) \left(\theta_n^{-2}\right)_{x_j}\right)_{x_i}\right| \leq C \lambda^2\gamma_0^2\theta^{-2}$. Thanks to \Cref{lemma2},  this yields
	\begin{eqnarray}\label{y_eps}
	\begin{array}{rl}
		&\displaystyle \int_Q  \theta_n^{-2} \gamma_0^{-2}\left| \nabla
		y_{\varepsilon}\right|^2  dxdt\\
		&\displaystyle \leq  e^{C\lambda} \left|y_0\right|_{L^2(\Omega)}^2
		+ \int_Q \xi_0 \gamma_0^{-2}\theta^{-2} \left|u_\varepsilon\right|
		 \left|y_\varepsilon\right| dxdt
		+ C  \lambda^2 \int_Q \theta^{-2} \left|y_\varepsilon\right|^2 dxdt\\
		&\displaystyle \leq e^{C\lambda} \left|y_0\right|_{L^2(\Omega)}^2
		+ C \lambda^{-2} \int_Q \theta^{-2}\gamma_0^{-3} \left|u_{\varepsilon}\right|^2 dxdt
		+ C \lambda^2 \int_Q  \theta^{-2} \left|y_{\varepsilon}\right|^2 dxdt\\
		&\displaystyle \leq C e^{C\lambda} \left|y_0\right|_{L^2(\Omega)}^2,
	\end{array}
	\end{eqnarray}
	for some positive constant $C.$
	Thus, the sequence $\left\{\theta_n^{-1} \gamma_0^{-1} \nabla
	y_{\varepsilon}\right\}_{n=1}^\infty$ is uniformly bounded and
	 converges weakly in $(L^2(Q))^n$. Moreover, $$\lim_{n\to\infty}\theta_n^{-1} \gamma_0^{-1} \nabla y_{\varepsilon} = \theta^{-1} \gamma_0^{-1} \nabla y_{\varepsilon}\quad \text{ a.e.}\  \text{ in } Q.$$
	By the uniqueness of the weak limit, we conclude that as $n\rightarrow\infty$,
	\begin{align*}
		\theta_n^{-1} \gamma_0^{-1} \nabla y_{\varepsilon} \rightarrow
		\theta^{-1} \gamma_0^{-1} \nabla y_{\varepsilon},\quad \text{ weakly  in } (L^2(Q))^n,
	\end{align*}
	and consequently,
	\begin{align*}
		\int_Q \theta^{-2} \gamma_0^{-2}\left| \nabla y_{\varepsilon}\right|^2 dxdt
		\leq \liminf\limits_{n\rightarrow\infty} \int_Q \theta_n^{-2} \gamma_0^{-2}\left| \nabla y_{\varepsilon}\right|^2 dxdt.
	\end{align*}
	Finally, from \eqref{y_eps},  we establish that
$\displaystyle
		\int_Q \theta^{-2} \gamma_0^{-2}\left| \nabla y_{\varepsilon}\right|^2  dxdt
		\leq C e^{C\lambda}  \left|y_0\right|_{L^2(\Omega)}^2.
$ Similarly,  we obtain
	$\displaystyle
		\int_Q \theta^{-2} \gamma_0^{-2}\left| \nabla w_{\varepsilon}\right|^2 dxdt
		\leq C e^{C\lambda}  \left|y_0\right|_{L^2(\Omega)}^2.
$
	Combining these two estimates yields  the desired result \eqref{grad_est}.
\end{proof}

Now, we are in a position to  prove \Cref{t3}.

\begin{proof}[Proof of \Cref{t3}]
For any $z\in K$, the key is to prove that the controls $u_\varepsilon$ and $v_\varepsilon$
obtained  in \Cref{lemma2}  belong to     $C^{1+\delta, \frac{1+\delta}{2}}(\overline{Q})$
 and to establish a uniform  estimate for
$|u_\varepsilon-v_\varepsilon|_{C^{1+\delta, \frac{1+\delta}{2}}(\overline{Q})}$.
As before, $C$ denotes a  generic positive constant that may vary from line to line,
 is independent of $\lambda$, and depends only on $\rho_0$, $\Omega$, $\omega_0$, $n$, $T$, $\lambda_0$, $\mu_0$, $\Lambda$ and $\mu$.

\ss

{\bf Step 1. }  Let $\{s_k\}_{k\in\dbN}$ be a strictly
increasing sequence satisfying  $0<\lambda/8<s_k<\lambda/4$, and set $\beta_0(t) = \left(e^{6\mu}-\mu e^{12\mu}\right)\gamma_0(t).$
For each $k\in\mathbb N$, $\varepsilon>0$ and $\lambda\geq 1$, we define
\begin{eqnarray*}
	U_\varepsilon^k := e^{(\lambda+s_k)\beta_0(t)}\gamma_0^3(t) p_\varepsilon\quad \mbox{ and }\quad   V_\varepsilon^k := e^{(\lambda+s_k)\beta_0(t)}\gamma_0^3(t) q_\varepsilon.
\end{eqnarray*}
Using standard arguments on \eqref{54} and \eqref{56},
 $U_\varepsilon^k$ satisfies
\begin{eqnarray}\label{58}
\left\{
\begin{array}{ll}
U^k_{\varepsilon,  t}+\sum\limits_{i, j=1}^n \left(a^{ij}(z, \nabla z)U^k_{\varepsilon, x_i}\right)_{x_j}
-\tilde{a}(z)U^k_\varepsilon+\mbox{div}\left(\tilde{B}(z)U^k_\varepsilon\right) =
\tilde g^k_\varepsilon+g_\varepsilon^k  &\mbox{ in   }Q,\\
U^k_\varepsilon=0  &\mbox{ on  }\Sigma,\\
U^k_\varepsilon(x, T)=0
&\mbox{ in  }\Omega,
\end{array}
\right.
\end{eqnarray}
where $\tilde g^k_\varepsilon=e^{(\lambda+s_k)\beta_0}\gamma_0^3 \theta^{-2} y_\varepsilon$ and $g_\varepsilon^k=[e^{(\lambda+s_k)\beta_0}\gamma_0^3]_t \, p_\varepsilon$.
Similarly,   $V_\varepsilon^k$  satisfies
\begin{eqnarray}\label{59}
\left\{
\begin{array}{ll}
V_{\varepsilon, t}^k+\sum\limits_{i, j=1}^n \left(\tilde{a}^{ij} V^k_{\varepsilon, x_i}\right)_{x_j}=\tilde h^k_\varepsilon+h_\varepsilon^k &\mbox{ in   }Q,\\
V_\varepsilon^k=0  &\mbox{ on  }\Sigma,\\
V_\varepsilon^k(x, T)=0  &\mbox{ in  }\Omega,
\end{array}
\right.
\end{eqnarray}
where $\tilde h^k_\varepsilon=e^{(\lambda+s_k)\beta_0}\gamma_0^3
\theta^{-2} w_\varepsilon$ and $h_\varepsilon^k=[e^{(\lambda+s_k)\beta_0}\gamma_0^3]_t\, q_\varepsilon$.

 We  shall prove that  for a sufficiently large $k\in\dbN$,  $U^k_\varepsilon,  V^k_\varepsilon\in
C^{1+\delta, \frac{1+\delta}{2}}(\overline{Q})$ and establish   a uniform  estimate for
 $|U^k_\varepsilon-V^k_\varepsilon|
_{C^{1+\delta, \frac{1+\delta}{2}}(\overline{Q})}$.  To this end,
we set
$$W^k_\varepsilon=U_\varepsilon^k-V_\varepsilon^k\quad\quad\mbox{and}\quad\quad
\widehat{a}^{i j}(z, \nabla z)=a^{i j}(z, \nabla z)-\tilde{a}^{i j}.$$
Then  $W^k_\varepsilon$ satisfies
\begin{eqnarray}\label{510}
\left\{
\begin{array}{ll}
W_{\varepsilon, t}^k+\sum\limits_{i, j=1}^n \left(\tilde{a}^{ij} W^k_{\varepsilon, x_i}\right)_{x_j}=
-\sum\limits_{i, j=1}^n \left(\widehat{a}^{ij}(z, \nabla z)U^k_{\varepsilon, x_i}\right)_{x_j}
+\tilde{a}(z) U^k_\varepsilon-\mbox{div}\left(\tilde{B}(z)U^k_\varepsilon\right)& \\
\quad\quad\quad\quad\quad\quad\quad\quad\quad\quad\quad\quad+g^k_\varepsilon-h^k_\varepsilon+\tilde g^k_\varepsilon-\tilde h^k_\varepsilon
 &\mbox{ in   }Q,\\
W_\varepsilon^k=0  &\mbox{ on  }\Sigma,\\
W_\varepsilon^k(x, T)=0  &\mbox{ in  }\Omega.
\end{array}
\right.
\end{eqnarray}

{\bf Step 2. } Note that $\beta_0\leq \alpha$ in $Q$, and therefore, for any $k\in\mathbb N$, we have
\begin{equation}\label{old+}
e^{2(\lambda+s_k)\beta_0}
= e^{2s_k (e^{6\mu}-\mu e^{12\mu})\gamma_0}\,
e^{2\lambda(e^{6\mu}-\mu e^{12\mu})\gamma_0}
\leq e^{2s_k (e^{6\mu}-\mu e^{12\mu})\gamma_0} \theta^2.
\end{equation}

Consider the case $k=1$.
Since $\mu$ is a fixed constant,   $\left|\gamma_0'(t)\right| \leq C \lambda \gamma_0^2(t)$ for $t\in (0,T)$ due to \eqref{gamma_prop}. Then,  we have
\begin{align*}
	\left|g_\varepsilon^1\right|^2_{L^2(Q)} =  \int_Q\big|[e^{(\lambda+s_1)\beta_0}\gamma_0^3
	]_t\big|^2p_\varepsilon^2 dxdt
	\leq  C \int_Q  \lambda^4 \gamma_0^{10} e^{2(\lambda+s_1)\beta_0}
	p_\varepsilon^2dxdt.
\end{align*}
Further, owing to \eqref{old+} and  $s_1>\frac{\lambda}{8}$, we note
\begin{align*}
	 \lambda^4 \gamma_0^{10} e^{2(\lambda+s_1)\beta_0}
	 \leq \lambda^4 \gamma_0^{10}e^{2s_1 (e^{6\mu}-\mu e^{12\mu})\gamma_0} \theta^2
	 \leq C \theta^2.
\end{align*}
Substituting this into the previous estimate of $\left|g_\varepsilon^1\right|^2_{L^2(Q)},$
 and  combining it with  \eqref{old3} and \eqref{55} yields
\begin{align*}
	\left|g_\varepsilon^1\right|^2_{L^2(Q)}
	\leq C\int^T_0\int_{\omega_0}  \theta^{2}\gamma_0^3 p_\varepsilon^2dxdt
	+C\int_Q \theta^{-2} y_\varepsilon^2 dxdt\leq  Ce^{C\lambda}|y_0|^2_{L^2(\Omega)}.
\end{align*}
Similarly, we estimate
$
	\left|h_\varepsilon^1\right|^2_{L^2(Q)}
	\leq  Ce^{C\lambda}|y_0|^2_{L^2(\Omega)}.
$
Using the estimates \eqref{old*} and \eqref{old+}, the $L^2$-norms of $g^1_\varepsilon$ and $h^1_\varepsilon$ are bounded by
\begin{eqnarray*}
&&\left|\tilde g^1_\varepsilon\right|^2_{L^2(Q)} + \left|\tilde h^1_\varepsilon\right|^2_{L^2(Q)}
\leq C\int_Q e^{2(\lambda+s_1)\beta_0}\gamma_0^6 \theta^{-4}(y^2_\varepsilon+w^2_\varepsilon)dxdt\\
&&\leq C\int_Q e^{2s_1(e^{6\mu}-\mu e^{12\mu})\gamma_0} \gamma_0^6\theta^{-2}
(y^2_\varepsilon+w^2_\varepsilon)dxdt \leq C\int_Q \theta^{-2} (y^2_\varepsilon+w^2_\varepsilon)dxdt
\leq Ce^{C\lambda}|y_0|^2_{L^2(\Omega)}.
\end{eqnarray*}
Applying
 the $L^p$-estimates for  linear parabolic equations (e.g. \cite[Lemma 4.1]{Liu}) to
   (\ref{58}) and
(\ref{59}) for $k=1$,  we obtain
  $U_\varepsilon^1,  V_\varepsilon^1\in W^{2, 1}_2(Q)$, and
\begin{equation}\label{511}
\left|U_\varepsilon^1\right|_{W^{2, 1}_2(Q)}\!+\!\left|V_\varepsilon^1\right|_{W^{2, 1}_2(Q)}
\!\leq\!  C \left(|\tilde g_\varepsilon^1|_{L^2(Q)}\!+\!|g_\varepsilon^1|_{L^2(Q)}\!+\!
|\tilde h_\varepsilon^1|_{L^2(Q)}\!+\!|h_\varepsilon^1|_{L^2(Q)}\right)
\!\leq\!
 Ce^{C\lambda}|y_0|_{L^2(\Omega)}.
\end{equation}

To estimate  $\left|W_\varepsilon^1\right|_{W^{2,1}_2(Q)}$,
 we first bound the coefficients. Since $\nabla f(0, {\bf 0})={\bf0}$,  we have
$$
\displaystyle \left|\tilde a(z)\right|_{C(\overline{Q})}
= \left|\int^1_0 f_s(\tau z, \tau\nabla z)d\tau-f_s(0, {\bf 0})\right|_{C(\overline{Q})}
\leq C|z|_{C^{1, 0}(\overline{Q})}.
$$
Similarly,
$$
\displaystyle \left|\tilde{B}(z)\right|_{(C(\overline{Q}))^n}
\leq C |z|_{C^{1, 0}(\overline{Q})}\quad\mbox{and}\quad
\displaystyle \left|\mbox{div} \tilde{B}(z)\right|_{C(\overline{Q})}
\leq C|z|_{C^{2, 0}(\overline{Q})}.
$$
Since $\widehat{a}^{i j}(0, {\bf 0})=0$, we also find   $
|\widehat{a}^{i j}(z, \nabla z)|\leq C|z|_{C^{1, 0}(\overline{Q})}.
$

Applying the $L^p$-estimates for linear parabolic equations once again to \eqref{510} for $k=1$,
and leveraging \eqref{old+} and \eqref{511} leads to
\begin{eqnarray*}
\begin{array}{ll}
&\displaystyle \left|W_\varepsilon^1\right|^2_{W^{2, 1}_2(Q)}\\[2mm]
&\leq C\big|\sum\limits_{i, j=1}^n (\widehat{a}^{ij}(z, \nabla z)U^1_{\varepsilon, x_i})_{x_j}\big|^2_{L^2(Q)}\\[2mm]
&\quad+ C|\tilde{a}(z) U_\varepsilon^1|^2_{L^2(Q)}+
C|\mbox{div} (\tilde{B}(z)U_\varepsilon^1)|^2_{L^2(Q)}
+C|\tilde g^1_\varepsilon-\tilde h^1_\varepsilon|^2
_{L^2(Q)}+C|g^1_\varepsilon-h^1_\varepsilon|^2
_{L^2(Q)}\\[2mm]
&\leq C\sum\limits_{i, j=1}^{n} |\widehat{a}^{i j}(z, \nabla z)|^2_{C(\overline{Q})}
|U^1_\varepsilon|^2_{W^{2, 0}_2(Q)}+C|z|^2_{C^{2, 0}(\overline{Q})}|U^1_\varepsilon|^2_{W^{1, 0}_2(Q)}\\[2mm]
&\displaystyle\quad+C|\tilde{a}(z)|^2_{C(\overline{Q})}|U^1_\varepsilon|^2_{L^2(Q)}
+C|\tilde{B}(z)|^2_{(C(\overline{Q}))^n}|\nabla U^1_\varepsilon|^2_{(L^2(Q))^n}
+C|\mbox{div} \tilde{B}(z)|^2_{C(\overline{Q})} |U^1_\varepsilon|^2_{L^2(Q)}\\[2mm]
&\displaystyle\quad+C\int_Q \big|[e^{(\lambda+s_1)\beta_0}\gamma_0^3]_t(p_\varepsilon-q_\varepsilon)\big|^2dxdt
+C\int_Q \big|e^{(\lambda+s_1)\beta_0}\gamma_0^3\theta^{-2}(y_\varepsilon-w_\varepsilon)\big|^2dxdt\\[2mm]
&\leq C |z|^2_{C^{1, 0}(\overline{Q})} |U^1_\varepsilon|^2_{W^{2, 0}_2(Q)}+
C|z|^2_{C^{2, 0}(\overline{Q})}|U^1_\varepsilon|^2_{W^{1, 0}_2(Q)}\\[2mm]
&\quad\displaystyle + C \left[\int_{Q}
\lambda^4  e^{2(\lambda+s_1)\beta_0}\gamma_0^{10}
(p_\varepsilon-q_\varepsilon)^2dxdt+
\int_{Q} e^{2(\lambda+s_1)\beta_0} \gamma_0^6 \theta^{-4}
(y_\varepsilon-w_\varepsilon)^2dxdt\right].
\end{array}
\end{eqnarray*}
Note that
$$
 \gamma_0^6 e^{2(\lambda+s_1)\beta_0} \theta^{-4}
 \leq \gamma_0^{7} e^{2s_1(e^{6\mu}-\mu e^{12\mu})\gamma_0}
\cdot \gamma_0^{-1} \theta^{-2}
\leq C \gamma_0^{-1}\theta^{-2},
$$
and
$$
\lambda^4 \gamma_0^{10} e^{2(\lambda+s_1)\beta_0}
\leq \lambda^4 \gamma_0^{10} e^{\frac{\lambda}{4}(e^{6\mu}-\mu e^{12\mu})\gamma_0}\theta^2
\leq C \theta^2,
$$
Using these two estimates, we ultimately conclude
\begin{eqnarray}\label{Newadd1}
\begin{array}{ll}
&\displaystyle\left|W_\varepsilon^1\right|^2_{W^{2, 1}_2(Q)}\\[3mm]
&\leq Ce^{C\lambda}|z|^2_{C^{2, 0}(\overline{Q})} |y_0|^2_{L^2(\Omega)}\!+\! C\displaystyle\left[
\displaystyle\int_Q \theta^2 (p_\varepsilon-q_\varepsilon)^2dxdt\!+\!
\int_{Q}  \theta^{-2}\gamma_0^{-1}
(y_\varepsilon-w_\varepsilon)^2dxdt
\right].
\end{array}
\end{eqnarray}

{\bf Step 3. } In this step,  we  estimate the last two integrals in (\ref{Newadd1}).
To this end,
by  (\ref{54})  and  (\ref{56}), $p_\varepsilon-q_\varepsilon$ satisfies the following system:
\begin{eqnarray*}
\left\{
\begin{array}{ll}
(p_\varepsilon-q_\varepsilon)_t+\sum\limits_{i, j=1}^n (\tilde{a}^{ij} (p_\varepsilon-q_{\varepsilon})_{x_i})_{x_j}
=-\sum\limits_{i, j=1}^n (\widehat{a}^{ij}(z, \nabla z)p_{\varepsilon, x_i})_{x_j}&\\
\quad\quad\quad\quad\quad\quad\quad\quad\quad\quad\quad\quad\quad\quad\quad\quad+\tilde{a}(z)p_\varepsilon-\mbox{div}(\tilde{B}(z)p_\varepsilon)+\theta^{-2}(y_\varepsilon-
w_\varepsilon) &\mbox{ in   }Q,\\
p_\varepsilon-q_\varepsilon=0  &\mbox{ on  }\Sigma.
\end{array}
\right.
\end{eqnarray*}
Setting $\tilde{\eta}_\varepsilon= \gamma_0^{-1}(p_\varepsilon-q_\varepsilon)$, we find that
  $\tilde{\eta}_\varepsilon$ satisfies
\begin{eqnarray*}
&&\tilde{\eta}_{\varepsilon, t}+\sum\limits_{i, j=1}^n (\tilde{a}^{ij} \tilde{\eta}_{\varepsilon, x_i})_{x_j}\\[-2mm]
&&=-\gamma_0^{-1}\sum\limits_{i, j=1}^n (\widehat{a}^{ij}(z, \nabla z)p_{\varepsilon, x_i})_{x_j}
+\gamma_0^{-1}\tilde{a}(z)p_\varepsilon-\gamma_0^{-1}\mbox{div}(\tilde{B}(z)p_\varepsilon)+\gamma_0^{-1}\theta^{-2}(y_\varepsilon-
w_\varepsilon)\\[-2mm]
&&\quad- \gamma_0^{-2}\gamma'_0(p_\varepsilon-q_\varepsilon).
\end{eqnarray*}
Similarly, $y_\varepsilon-w_\varepsilon$ satisfies
$$(y_\varepsilon-w_\varepsilon)_t-\sum\limits_{i, j=1}^n (\tilde{a}^{ij}
 (y_\varepsilon-w_\varepsilon)_{x_i})_{x_j}=\sum\limits_{i, j=1}^n (\widehat{a}^{ij}(z, \nabla z)y_{\varepsilon, x_i})_{x_j}
-\tilde{a}(z)y_\varepsilon-\tilde{B}(z)\cdot \nabla y_\varepsilon+\xi_0(u_\varepsilon-v_\varepsilon).
$$
By utilizing the duality relation   between the equations  governed by  $y_\varepsilon-w_\varepsilon$ and $\tilde{\eta}_{\varepsilon}$,  we obtain
\begin{align}\label{identity}
	& \int_Q \xi^2_{0} \theta^{2}\lambda^3\gamma_0^{2}
	(p_\varepsilon-q_\varepsilon)^2dxdt
	+\int_Q \theta^{-2}\gamma_0^{-1}(y_\varepsilon-w_\varepsilon)^2dxdt	\nonumber\\
	& = \int_Q\sum\limits_{i, j=1}^n \gamma_0^{-1}\, \widehat{a}^{i j}(z, \nabla z)
	\left(p_{\varepsilon, x_i} w_{\varepsilon, x_j} - q_{\varepsilon, x_i} y_{\varepsilon, x_j}\right)dxdt \nonumber\\
	& \quad+ \int_Q \gamma_0^{-1}\tilde{a}(z) \left( p_\varepsilon - q_\varepsilon\right) y_\varepsilon dxdt
	- \int_Q \gamma_0^{-1}\tilde{a}(z) \left( y_\varepsilon - w_\varepsilon\right) p_\varepsilon dxdt \nonumber\\
	&\quad + \int_Q \gamma_0^{-1} \left( p_\varepsilon - q_\varepsilon\right) \tilde{B}(z)\cdot \nabla y_\varepsilon dxdt
	+ \int_Q \gamma_0^{-1} \operatorname{div}\left(\tilde{B}(z) p_\varepsilon\right) \left( y_\varepsilon - w_\varepsilon\right)dxdt \nonumber\\
	& \quad+ \int_Q \gamma_0^{-2}\gamma_0' \left( p_\varepsilon - q_\varepsilon\right) \left( y_\varepsilon - w_\varepsilon\right)dxdt \nonumber\\
	&=\sum\limits_{i=1}^{4} I_i,
\end{align}
where $I_i$ denotes the $i$-th line on the right-hand side of the identity  above.
We now estimate each of the terms $I_i$, $i=1,\ldots,4$ as follows.
For the first term $I_1$, we have
\begin{eqnarray}\label{I1_est1}
\begin{array}{rl}
	I_1
	&\displaystyle = \int_Q\sum\limits_{i, j=1}^n \gamma_0^{-1} \,\widehat{a}^{i j}(z, \nabla z) \Big(
	p_{\varepsilon, x_i} \left(w_{\varepsilon}-y_{\varepsilon}\right)_{x_j} + \left( p_{\varepsilon} - q_{\varepsilon}\right)_{x_i} y_{\varepsilon,x_j}\Big) dxdt \\
	&\displaystyle = -\int_Q \sum\limits_{i, j=1}^n \gamma_0^{-1} \Big( \left(\widehat{a}^{i j}(z, \nabla z)\right)_{x_j} p_{\varepsilon,x_i} + \widehat{a}^{i j}(z, \nabla z) p_{\varepsilon,x_ix_j} \Big) \left(y_\varepsilon-w_\varepsilon\right) dxdt  \\
	& \displaystyle\hspace{5mm} + \int_Q\sum\limits_{i, j=1}^n \gamma_0^{-1} \,\widehat{a}^{i j}(z, \nabla z) \left( p_{\varepsilon} - q_{\varepsilon}\right)_{x_i} y_{\varepsilon,x_j} dxdt \\
	&\displaystyle\leq C \sqrt{\lambda} \left|z\right|_{C^{2,0}(\overline{Q})} \left(
	\left|\theta\, \nabla p_{\varepsilon}\right|_{L^2(Q)}
	+ \sum_{i,j=1}^n \left|\lambda^{-1/2}\gamma_0^{-1/2} \theta p_{\varepsilon,x_ix_j} \right|_{L^2(Q)}
	\right) \left|\theta^{-1}\gamma_0^{-1/2}(y_\varepsilon-w_\varepsilon)\right|_{L^2(Q)} \\
	&\displaystyle\hspace{5mm} + C \left|z\right|_{C^{1,0}(\overline{Q})}
	\left| \theta \nabla(p_{\varepsilon}-q_{\varepsilon})\right|_{L^2(Q)} \left|\theta^{-1}\gamma_0^{-1}\nabla y_\varepsilon\right|_{L^2(Q)}.
	\end{array}
\end{eqnarray}
Applying Corollary \ref{corollary4*} to the equation satisfied  by $p_\varepsilon-q_\varepsilon$ with $m=-1$, and subsequently  using \eqref{old3} and \eqref{old*}, we  obtain
\begin{eqnarray}\label{pq_diff}
\begin{array}{rl}
	&\displaystyle \int_Q \theta^2 \left|\nabla(p_\varepsilon-q_\varepsilon)\right|^2 dxdt
	+ \int_Q \theta^2 \lambda^2\gamma_0^2\left|p_\varepsilon-
	q_\varepsilon\right|^2 dxdt \\
	&\displaystyle\leq C\int^T_0\int_{\omega_0}   \theta^2\lambda^2\gamma_0^2 (p_\varepsilon-q_\varepsilon)^2 dxdt
	+C\int_Q \theta^2 \lambda^{-1} \gamma_0^{-1} \left(
	\left|\tilde{a}(z) p_\varepsilon\right|^2+\left|\theta^{-2}(y_\varepsilon-w_\varepsilon)\right|^2\right)dxdt \\
	&\displaystyle\hspace{5cm} + C \int_Q \theta^2 \lambda \gamma_0 \left(
	\sum\limits_{i, j=1}^{n}\left|\widehat{a}^{i j}(z, \nabla z) p_{\varepsilon, x_i}\right|^2+\left|\tilde{B}(z)p_\varepsilon\right|^2
	\right)dxdt \\
	&\displaystyle\leq C \int^T_0\int_{\omega_0}   \theta^2\lambda^2\gamma_0^2 (p_\varepsilon-q_\varepsilon)^2 dxdt
	+C \int_Q \theta^{-2} \lambda^{-1} \gamma_0^{-1}(y_\varepsilon-w_\varepsilon)^2dxdt
	+ Ce^{C\lambda} |z|^2_{C^{1, 0}(\overline{Q})}|y_0|^2_{L^2(\Omega)}.
	\end{array}
\end{eqnarray}
By employing  the estimates \eqref{old*} and \eqref{old3} for the first  component,
 and the estimates \eqref{grad_est} and \eqref{pq_diff} for the second, we ultimately conclude
\begin{eqnarray}\label{I1_est}
\begin{array}{rl}
	&\displaystyle I_1 \leq Ce^{C\lambda} \left|z\right|_{C^{2,0}(\overline{Q})} |y_0|_{L^2(\Omega)}
	\left( \int^T_0\int_{\omega_0}   \theta^2\lambda^2\gamma_0^2 (p_\varepsilon-q_\varepsilon)^2 dxdt
	+\int_Q \theta^{-2} \gamma_0^{-1}
	(y_\varepsilon-w_\varepsilon)^2dxdt \right)^{1/2} \\
	& \displaystyle\hspace{5mm} + C e^{C\lambda} |z|^2_{C^{2, 0}(\overline{Q})}|y_0|^2_{L^2(\Omega)} \\
	&\displaystyle \leq C e^{C\lambda} |z|^2_{C^{2, 0}(\overline{Q})}|y_0|^2_{L^2(\Omega)}
	+ \epsilon \left( \int^T_0\int_{\omega_0}   \theta^2\lambda^2\gamma_0^2 (p_\varepsilon-q_\varepsilon)^2 dxdt
	+\int_Q \theta^{-2} \gamma_0^{-1} (y_\varepsilon-w_\varepsilon)^2dxdt \right),
\end{array}
\end{eqnarray}
for any $\epsilon>0$.
Similarly, utilizing  H\"older's inequality alongside
the estimates \eqref{old*}, \eqref{old3} and \eqref{pq_diff},
$I_2$ can be bounded by:
\begin{align*}
	I_2 & \leq C \left|z\right|_{C^{1,0}(\overline{Q})} \Big( \left|\theta^{-1}y_\varepsilon\right|_{L^2(Q)} \left|\theta (p_\varepsilon-q_\varepsilon)\right|_{L^2(Q)} + \left|\theta^{-1} \gamma_0^{-1/2} \left(y_\varepsilon-w_\varepsilon\right)\right|_{L^2(Q)} \left|\theta p_\varepsilon\right|_{L^2(Q)} \Big)\\
	& \leq C e^{C\lambda} |z|^2_{C^{2, 0}(\overline{Q})}|y_0|^2_{L^2(\Omega)}
	+ \epsilon \left( \int^T_0\int_{\omega_0}   \theta^2\lambda^2\gamma_0^2 (p_\varepsilon-q_\varepsilon)^2 dxdt
	+\int_Q \theta^{-2} \gamma_0^{-1} (y_\varepsilon-w_\varepsilon)^2dxdt \right).
\end{align*}
For the term$I_3$, we use the estimates \eqref{old*}, \eqref{old3}, \eqref{grad_est} and \eqref{pq_diff} to get
\begin{align*}
	I_3 & \leq \left|z\right|_{C^{2,0}(\overline{Q})} \bigg( \left|\theta^{-1}\gamma_0^{-1}\nabla y_\varepsilon\right|_{L^2(Q)} \left|\theta (p_\varepsilon-q_\varepsilon)\right|_{L^2(Q)}\\
	&\hspace{3cm} + \left( \left|\theta p_\varepsilon\right|_{L^2(Q)} + \left|\theta\nabla p_\varepsilon\right|_{L^2(Q)} \right) \left|\theta^{-1} \gamma_0^{-1/2} \left(y_\varepsilon-w_\varepsilon\right)\right|_{L^2(Q)} \bigg)\\
	& \leq C e^{C\lambda} |z|^2_{C^{2, 0}(\overline{Q})}|y_0|^2_{L^2(\Omega)}
	+ \epsilon \left( \int^T_0\int_{\omega_0}   \theta^2\lambda^2\gamma_0^2 (p_\varepsilon-q_\varepsilon)^2 dxdt
	+\int_Q \theta^{-2} \gamma_0^{-1} (y_\varepsilon-w_\varepsilon)^2dxdt \right).
\end{align*}
Lastly,   we use  H\"older's inequality and Young's inequality for $I_4$  to get
\begin{align*}
	I_4 & \leq \lambda  \left|\theta \gamma_0^{1/2}(p_\varepsilon-q_\varepsilon)\right|_{L^2(Q)}  \left|\theta^{-1} \gamma_0^{-1/2} (y_\varepsilon-w_\varepsilon)\right|_{L^2(Q)}\\
	& \leq C \lambda^2 \int_Q \theta^2\gamma_0 \left|p_\varepsilon-q_\varepsilon\right|^2 dxdt
	+ \epsilon \int_Q  \theta^{-2} \gamma_0^{-1} (y_\varepsilon-w_\varepsilon)^2dxdt.
\end{align*}

Substituting the estimates for $I_1$--$I_4$ into
 \eqref{identity} and choosing $\epsilon$ to be sufficiently small, we deduce
\begin{eqnarray}\label{imp_est1}
\begin{array}{rl}
	&\displaystyle\int_Q \xi^2_{0} \theta^{2}\lambda^3\gamma_0^{2}
	(p_\varepsilon-q_\varepsilon)^2dxdt
	+\int_Q \theta^{-2}\gamma_0^{-1}(y_\varepsilon-w_\varepsilon)^2dxdt\\
	&\displaystyle\leq C e^{C\lambda} |z|^2_{C^{2, 0}(\overline{Q})}|y_0|^2_{L^2(\Omega)}
	+ C \lambda^2 \int_Q \theta^2\gamma_0 \left|p_\varepsilon-q_\varepsilon\right|^2 dxdt.
\end{array}
\end{eqnarray}
Thanks to \eqref{pq_diff}, we have
\begin{align*}
	&\int_Q \theta^2 \lambda^2\gamma_0^2\left|p_\varepsilon-q_\varepsilon\right|^2 dxdt\\
	&\leq \frac{C}{\lambda} \left[ \int_Q \xi_0^2  \theta^2\lambda^3\gamma_0^2 (p_\varepsilon-q_\varepsilon)^2 dxdt
	+\int_Q \theta^{-2} \gamma_0^{-1}(y_\varepsilon-w_\varepsilon)^2dxdt \right]
	+ Ce^{C\lambda} |z|^2_{C^{1, 0}(\overline{Q})}|y_0|^2_{L^2(\Omega)},
\end{align*}
which, together with \eqref{imp_est1},  gives
\begin{align*}
	\int_Q \theta^2 \lambda^2\gamma_0^2\left|p_\varepsilon-q_\varepsilon\right|^2 dxdt
	\leq C e^{C\lambda} |z|^2_{C^{2, 0}(\overline{Q})}|y_0|^2_{L^2(\Omega)}
	+ C \lambda \int_Q \theta^2\gamma_0 \left|p_\varepsilon-q_\varepsilon\right|^2 dxdt.
\end{align*}
By selecting   $\lambda$ sufficiently large,
we can  absorb the last term on the right-hand side,  which yields
\begin{align*}
		\int_Q \theta^2 \lambda^2\gamma_0^2\left|p_\varepsilon-q_\varepsilon\right|^2 dxdt
	\leq C e^{C\lambda} |z|^2_{C^{2, 0}(\overline{Q})}|y_0|^2_{L^2(\Omega)}.
\end{align*}
Substituting the above estimate  into \eqref{imp_est1}, we  conclude
\begin{align}\label{diff_est_final}
	\int_Q \xi^2_{0} \theta^{2}\lambda^3\gamma_0^{2} (p_\varepsilon-q_\varepsilon)^2dxdt +
	\int_Q \theta^{-2}\gamma_0^{-1}(y_\varepsilon-w_\varepsilon)^2dxdt
	\leq C e^{C\lambda} |z|^2_{C^{2, 0}(\overline{Q})}|y_0|^2_{L^2(\Omega)}.
\end{align}
Finally, integrating  the last two estimates into \eqref{Newadd1} provides the bound:
\begin{align}\label{Newadd18}
	\left|W_\varepsilon^1\right|^2_{W^{2, 1}_2(Q)}
	\leq Ce^{C\lambda} \left|z\right|^2_{C^{2, 0}(\overline{Q})} |y_0|^2_{L^2(\Omega)}.
\end{align}
From this point onward, we fix $\lambda$ sufficiently large so that the previous analysis holds. Consequently, the generic constant $C$ appearing in the subsequent estimates may now depend on $\lambda$, which is fixed.

\ss

{\bf Step 4. }
Notice that the controls can be expressed as
\begin{equation}\label{ll}
u_\varepsilon=\xi_{0}\lambda^3 \theta^{2} e^{-(\lambda+s_k)\beta_0} U_\varepsilon^k
\quad  \mbox{  and  }\quad
v_\varepsilon=\xi_{0}\lambda^3 \theta^{2} e^{-(\lambda+s_k)\beta_0} V_\varepsilon^k.
\end{equation}
Hence,  by (\ref{511}), $u_\varepsilon, v_\varepsilon\in W^{2, 1}_2(Q)$ and
\begin{equation}\label{516}
\left|u_\varepsilon\right|_{W^{2, 1}_2(Q)} + \left|v_\varepsilon\right|_{W^{2, 1}_2(Q)}\leq
C\left|U_\varepsilon^1\right|_{W^{2, 1}_2(Q)} + C\left|V_\varepsilon^1\right|_{W^{2, 1}_2(Q)}
\leq C |y_0|_{L^2(\Omega)}.
\end{equation}

By (\ref{Newadd18}), (\ref{516}),
  and the Sobolev  embedding theorem,   for $r_1=2(n+2)/(n-2)$ $(\mbox{if } n>2)$  or
$r_1\geq 2$ $(\mbox{if } n\leq 2)$, we have
\begin{equation}\label{Newadd12}
\left|W_\varepsilon^1\right|^2_{L^{r_1}(Q)}\leq
C\left|W_\varepsilon^1\right|^2_{W^{2,  1}_2(Q)}
\leq C |z|^2_{C^{2, 0}(\overline{Q})} |y_0|^2_{L^2(\Omega)},
\end{equation}
and
\begin{equation}\label{Newadd13}
\left|u_\varepsilon\right|^2_{L^{r_1}(Q)} + \left|v_\varepsilon\right|^2_{L^{r_1}(Q)}\leq
C \left( \left|U^1_\varepsilon\right|^2_{L^{r_1}(Q)} + \left|V^1_\varepsilon\right|^2_{L^{r_1}(Q)}\right)
\leq C |y_0|^2_{L^2(\Omega)}.
\end{equation}

\medskip

For the case  $k=2$,
observe that
\begin{eqnarray*}
&&g_\varepsilon^2=
[e^{(\lambda+s_2)\beta_0}\gamma_0^3]_{t} p_\varepsilon
=[e^{(\lambda+s_2)\beta_0}\gamma_0^3]_{t}
e^{-(\lambda+s_1)\beta_0}\gamma_0^{-3} U_\varepsilon^1,\\[3mm]
&&h_\varepsilon^2=
[e^{(\lambda+s_2)\beta_0}\gamma_0^3]_{t} q_\varepsilon
=[e^{(\lambda+s_2)\beta_0}\gamma_0^3]_{t}
e^{-(\lambda+s_1)\beta_0}\gamma_0^{-3} V_\varepsilon^1,\\[3mm]
&&g_\varepsilon^2-h_\varepsilon^2=
[e^{(\lambda+s_2)\beta_0}\gamma_0^3]_{t} (p_\varepsilon-q_\varepsilon)
=[e^{(\lambda+s_2)\beta_0}\gamma_0^3]_{t}
e^{-(\lambda+s_1)\beta_0}\gamma_0^{-3} W_\varepsilon^1.
\end{eqnarray*}
Hence, by utilizing (\ref{Newadd12}) and (\ref{Newadd13}), we deduce
\begin{eqnarray}\label{517}
\begin{array}{rl}
&\displaystyle \left|g_\varepsilon^2-h_\varepsilon^2\right|^2_{L^{r_1}(Q)}\leq
C\left|W_\varepsilon^1\right|^2_{L^{r_1}(Q)}\leq  C
 |z|^2_{C^{2, 0}(\overline{Q})}
|y_0|^2_{L^2(\Omega)}, \\[3mm]
&\displaystyle\mbox{and}\quad
\left|g_\varepsilon^2\right|_{L^{r_1}(Q)} + \left|h_\varepsilon^2\right|_{L^{r_1}(Q)}
\leq C\left( \left|U^1_\varepsilon\right|_{L^{r_1}(Q)} + \left|V^1_\varepsilon\right|_{L^{r_1}(Q)}\right)
\leq C|y_0|_{L^2(\Omega)}.
\end{array}
\end{eqnarray}

Next, we  estimate  the terms
$\sum\limits_{i, j=1}^n (\widehat{a}^{i j}(z, \nabla z) U^2_{\varepsilon, x_i})_{x_j}$,
$\tilde{a}(z)U^2_\varepsilon-\mbox{div}(\tilde{B}(z)U^2_\varepsilon)$ and
$\tilde{g}_\varepsilon^2-\tilde{h}_\varepsilon^2$ in (\ref{510}) for $k=2$.  To this end,
we set
$
L(x, t)=e^{(\lambda+s_2)\beta_0(t)}\theta^{-2}(x, t)\gamma_0^3(t).
$
Direct substitution shows
$$\tilde g^2_\varepsilon=L y_\varepsilon, \quad
\tilde h^2_\varepsilon=L w_\varepsilon,\quad
L u_\varepsilon=\xi_{0}\lambda^3\gamma_0^3 e^{(s_2-s_1)\beta_0} U^1_\varepsilon\   \mbox{  and }\
L v_\varepsilon=\xi_{0}\lambda^3\gamma_0^3 e^{(s_2-s_1)\beta_0}  V^1_\varepsilon.
$$
Moreover,  $\tilde g^2_\varepsilon$ and $\tilde h^2_\varepsilon$
satisfy the following systems:
\begin{eqnarray}\label{519}
\left\{
\begin{array}{ll}
\tilde g^2_{\varepsilon, t}-\sum\limits_{i, j=1}^{n} \left(a^{i j}(z, \nabla z) \tilde{g}^2_{\varepsilon, x_i}\right)_{x_j}
+\tilde{a}(z)\tilde g^2_{\varepsilon}
+\tilde{B}(z)\cdot\nabla \tilde g^2_{\varepsilon}
=F_\varepsilon &\mbox{ in   }Q,\\[2mm]
\tilde g^2_\varepsilon=0  &\mbox{ on  }\Sigma,\\[2mm]
\tilde g^2_\varepsilon(x, 0)=L(x, 0)y_0(x)  &\mbox{ in  }\Omega,
\end{array}
\right.
\end{eqnarray}
and
\begin{eqnarray}\label{520}
\left\{
\begin{array}{ll}
\tilde h^2_{\varepsilon, t}-\sum\limits_{i, j=1}^{n} \left(\tilde{a}^{i j}\tilde{h}^2_{\varepsilon, x_i}\right)_{x_j} = G_\varepsilon  &\mbox{ in   }Q,\\[2mm]
\tilde h^2_\varepsilon=0  &\mbox{ on  }\Sigma,\\[2mm]
\tilde h^2_\varepsilon(x, 0)=L(x, 0)y_0(x)  &\mbox{ in  }\Omega,
\end{array}
\right.
\end{eqnarray}
where
\begin{align*}
F_\varepsilon & =
\xi_{0} L u_\varepsilon+L_t y_\varepsilon
-\sum_{i, j=1}^{n} a^{i j}(z, \nabla z)L_{x_i x_j} y_\varepsilon
-2\sum_{i, j=1}^{n} a^{i j}(z, \nabla z)L_{x_i} y_{\varepsilon, x_j}\\[-3mm]
&\quad\quad\quad-\sum_{i, j=1}^{n} \left(a^{i j}(z, \nabla z)\right)_{x_j} L_{x_i}y_\varepsilon
+\tilde{B}(z)\cdot\nabla L y_\varepsilon\\
&=\xi_0^2 \lambda^3\gamma_0^3 e^{(s_2-s_1)\beta_0}U^1_\varepsilon
+L_t y_\varepsilon-\sum_{i, j=1}^{n} a^{i j}(z, \nabla z)L_{x_i x_j}y_\varepsilon
-2\sum_{i, j=1}^{n} a^{i j}(z, \nabla z)L_{x_i}y_{\varepsilon, x_j}\\[-3mm]
&\quad\quad\quad - \sum_{i, j=1}^{n} \left(a^{i j}(z, \nabla z)\right)_{x_j} L_{x_i}y_\varepsilon
+\tilde{B}(z)\cdot\nabla L y_\varepsilon,
\end{align*}
and
\begin{align*}
G_\varepsilon
& =\xi_{0} L v_\varepsilon+L_t w_\varepsilon
-\sum_{i, j=1}^{n} \tilde{a}^{i j}L_{x_i x_j} w_\varepsilon-
2\sum_{i, j=1}^{n} \tilde{a}^{i j}L_{x_i} w_{\varepsilon, x_j}
\\[-3mm]
&=\xi_0^2 \lambda^3\gamma_0^3 e^{(s_2-s_1)\beta_0}V^1_\varepsilon
+L_t w_\varepsilon
-\sum_{i, j=1}^{n} \tilde{a}^{i j}L_{x_i x_j} w_\varepsilon-
2\sum_{i, j=1}^{n} \tilde{a}^{i j}L_{x_i} w_{\varepsilon, x_j}.
\end{align*}
Notice that the weight function $L$ and its derivatives satisfy
\begin{eqnarray*}
&&|L(x, t)|\leq C e^{s_2\beta_0}\theta^{-1}\gamma_0^3,\quad
|\nabla L(x, t)|\leq C e^{s_2\beta_0}\theta^{-1}\gamma_0^4,\\[2mm]
&&|L_{x_i x_j}(x, t)|+|L_t(x, t)|\leq C e^{s_2\beta_0}\theta^{-1}\gamma_0^5,\quad
i, j=1, 2, \cdots, n.
\end{eqnarray*}
Hence, there exists a positive constant $\hat{c}$ such that
$$
|L(x, t)|+|\nabla L(x, t)|+\sum_{i, j=1}^{n} |L_{x_i x_j}(x, t)|+|L_t(x, t)|\leq C
e^{-\hat{c}\gamma_0}\theta^{-1},
$$
and therefore,
\begin{equation}\label{Newadd28}
\left|F_\varepsilon\right|_{L^2(Q)}\leq
C\left|U^1_\varepsilon\right|_{L^2(Q)} + C \left|\theta^{-1}y_\varepsilon\right|_{L^2(Q)}
+C\left|e^{-\hat{c}\gamma_0}\theta^{-1}\nabla y_\varepsilon\right|_{L^2(Q)}.
\end{equation}

Multiplying the equation for $y_\varepsilon$ by $e^{-2\hat{c}\gamma_0}\theta^{-2}y_\varepsilon$ and integrating
 over $Q$, we obtain
\begin{eqnarray*}
&&\int_Q e^{-2\hat{c}\gamma_0}\theta^{-2}|\nabla y_\varepsilon|^2dxdt
\leq C|y_0|^2_{L^2(\Omega)}+C\int_Q \theta^{-2} y^2_\varepsilon dxdt
+C\int_Q \theta^{-2}e^{-4\hat{c}\gamma_0}\xi^2_0 u^2_\varepsilon dxdt\\
&&\leq C|y_0|^2_{L^2(\Omega)}+C\int_Q \theta^{-2} y^2_\varepsilon dxdt
+C\int_Q \theta^{-2}\gamma_0^{-3}\xi^2_0 u^2_\varepsilon dxdt.
\end{eqnarray*}
This inequality, in conjunction  with (\ref{old*}), confirms that
\begin{equation}\label{Newadd33}
	\displaystyle \int_Q e^{-2\hat{c}\gamma_0} \theta^{-2}|\nabla y_\varepsilon|^2dxdt
	\leq C|y_0|^2_{L^2(\Omega)}.
\end{equation}
Using this estimate alongside  \eqref{old*} and \eqref{511} in \eqref{Newadd28} gives
$\left|F_\varepsilon\right|_{L^2(Q)}\leq C|y_0|_{L^2(\Omega)}.$
Through analogous arguments,  we also obtain $\left|G_\varepsilon\right|_{L^2(Q)}\leq C|y_0|_{L^2(\Omega)}.$

Applying the $L^p$-estimates for linear parabolic equations  to (\ref{519}) and (\ref{520}),
and utilizing the Sobolev embedding theorem, we deduce
\begin{equation}\label{521}
\left|\tilde g^2_\varepsilon\right|_{L^{r_1}(Q)}\leq C
\left|\tilde g^2_\varepsilon\right|_{W^{2, 1}_{2}(Q)}\leq
C\left( \left|F_\varepsilon\right|_{L^2(Q)} + |y_0|_{C^{2}(\overline\Omega)}\right)
\leq C |y_0|_{C^{2}(\overline\Omega)},
\end{equation}
and
\begin{equation}\label{521*}\displaystyle
\left|\tilde h^2_\varepsilon\right|_{L^{r_1}(Q)}
\leq C \left|\tilde h^2_\varepsilon\right|_{W^{2, 1}_{2}(Q)}
\leq C\left( \left|G_\varepsilon\right|_{L^2(Q)} + |y_0|_{C^{2}(\overline\Omega)}\right)
\leq C |y_0|_{C^{2}(\overline\Omega)}.
\end{equation}

\smallskip

Furthermore,  setting $\tilde \eta^2_\varepsilon=\tilde g^2_\varepsilon-\tilde h^2_\varepsilon$,
we find that  $\tilde \eta^2_\varepsilon$ satisfies
\begin{align*}
&\tilde \eta^2_{\varepsilon, t}-
\sum_{i, j=1}^{n} \left(\tilde{a}^{i j} \tilde\eta^2_{\varepsilon, x_i}\right)_{x_j}\\[-2mm]
&=\xi_0^2 \lambda^3\gamma_0^3 e^{(s_2-s_1)\beta_0}W^1_\varepsilon
+L_t(y_\varepsilon-w_\varepsilon)-\sum_{i, j=1}^{n}  \tilde{a}^{i j}L_{x_i x_j}(y_\varepsilon-w_\varepsilon)-2\sum_{i, j=1}^{n}  \tilde{a}^{i j}L_{x_i}(y_\varepsilon-w_\varepsilon)_{x_j}\\[-2mm]
&\quad+
\sum_{i, j=1}^{n} \left(\widehat{a}^{i j}(z, \nabla z)\tilde{g}^2_{\varepsilon, x_i}\right)_{x_j}
-\tilde{a}(z) \tilde{g}^2_\varepsilon-\tilde{B}(z)\cdot\nabla \tilde{g}^2_\varepsilon
\\[-2mm]
&\quad-\sum_{i, j=1}^{n} \widehat{a}^{i j}(z, \nabla z)L_{x_i x_j} y_\varepsilon-2\sum_{i, j=1}^{n} \widehat{a}^{i j}(z, \nabla z)L_{x_i} y_{\varepsilon, x_j}
- \sum_{i, j=1}^{n} \left(a^{i j}(z, \nabla z)\right)_{x_j} L_{x_i} y_\varepsilon
+\tilde{B}(z)\cdot \nabla L y_\varepsilon\\
&:= \sum_{i=1}^3\tilde{I}_i.
\end{align*}
	Multiplying the equation governing  $\left(y_\varepsilon-w_\varepsilon\right)$ by $e^{-2\hat{c}\gamma_0}\theta^{-2}\left(y_\varepsilon-w_\varepsilon\right)$ and integrating over $Q$, we obtain
	\begin{align}\label{id1}
		& \int_Q e^{-2\hat{c}\gamma_0}\theta^{-2} \sum_{i,j} \tilde{a}^{ij} \left(y_\varepsilon-w_\varepsilon\right)_{x_i} \left(y_\varepsilon-w_\varepsilon\right)_{x_j} dxdt \nonumber\\
		& \leq \int_Q \left(e^{-2\hat{c}\gamma_0}\theta^{-2} \right)_t \left|y_\varepsilon-w_\varepsilon\right|^2dxdt
		+ \int_Q e^{-2\hat{c}\gamma_0}  \left|y_\varepsilon-w_\varepsilon\right| \, \sum_{i,j} \left| \left(\theta^{-2}\right)_{x_j} \tilde{a}^{ij} \left(y_\varepsilon-w_\varepsilon\right)_{x_i}\right|dxdt \nonumber\\
		&\hspace{5mm} + \int_Q e^{-2\hat{c}\gamma_0}
		 \left|y_\varepsilon-w_\varepsilon\right|\, \sum_{i,j}
		 \left| \left(\theta^{-2}\right)_{x_j} \widehat{a}^{ij}(z,\nabla z) y_{\varepsilon,x_i}\right|dxdt
		\nonumber\\
		&\hspace{5mm} + \int_Q e^{-2\hat{c}\gamma_0} \theta^{-2} \sum_{i,j} \left|\widehat{a}^{ij}(z,\nabla z)\left(y_\varepsilon-w_\varepsilon\right) _{x_j} y_{\varepsilon,x_i}\right|dxdt \nonumber\\
		&\hspace{5mm} + \int_Q \left|\tilde{a}(z)\right| \theta^{-2} e^{-2\hat{c}\gamma_0} \left|y_\varepsilon\right| \left|y_\varepsilon-w_\varepsilon\right|dxdt
		+ \int_Q  \theta^{-2} e^{-2\hat{c}\gamma_0} \left| \tilde{B}(z) \cdot \nabla y_\varepsilon\right| \left|y_\varepsilon-w_\varepsilon\right|dxdt \nonumber\\
		&\hspace{5mm} + \int_Q \xi_0 \theta^{-2} e^{-2\hat{c}\gamma_0} \left|u_\varepsilon-v_\varepsilon\right|  \left|y_\varepsilon-w_\varepsilon\right|dxdt\nonumber\\
		& := \sum_{i=1}^{7} J_i.
	\end{align}
	Since $\lambda,\mu$ are fixed, we have $\left|\left(e^{-2\hat{c}\gamma_0}\theta^{-2} \right)_t\right|\leq C \gamma_0^{-1}\theta^{-2}$ and $\left|\left(\theta^{-2}\right)_{x_j}\right| \leq C \theta^{-2}$.
	Employing  H\"older's inequality and the estimate \eqref{diff_est_final}, the terms
	$J_1, J_2,  J_5$ and $J_7$ can be bounded as:
	\begin{align*}
		J_1 + J_2 + J_5 + J_7 \leq C e^{C\lambda} |z|^2_{C^{2, 0}(\overline{Q})}|y_0|^2_{L^2(\Omega)}.
	\end{align*}
	Similarly, owing to the estimates \eqref{grad_est} and \eqref{diff_est_final}, we establish
	\begin{align*}
		J_3 + J_6 \leq C e^{C\lambda} |z|^2_{C^{2, 0}(\overline{Q})}|y_0|^2_{L^2(\Omega)}.
	\end{align*}
	For $J_4$,  employing
	 Young's inequality alongside the estimate \eqref{Newadd33} yields:
	\begin{align*}
		J_4
		& \leq |z|_{C^{1, 0}(\overline{Q})} \int_Q \theta^{-2} e^{-2\hat{c}\gamma_0} \left|\nabla y_\varepsilon\cdot \nabla(y_\varepsilon-w_\varepsilon)\right|dxdt \\
		& \leq \epsilon \int_Q \theta^{-2} e^{-2\hat{c}\gamma_0} \left|\nabla (y_\varepsilon-w_\varepsilon)\right|^2dxdt
		+ |z|^2_{C^{1, 0}(\overline{Q})} \int_Q \theta^{-2} e^{-2\hat{c}\gamma_0} \left|\nabla y_\varepsilon\right|^2dxdt\\
		& \leq \epsilon \int_Q \theta^{-2} e^{-2\hat{c}\gamma_0} \left|\nabla (y_\varepsilon-w_\varepsilon)\right|^2 dxdt
		+ C |z|^2_{C^{1, 0}(\overline{Q})} |y_0|^2_{L^2(\Omega)}.
	\end{align*}
	Finally,  substituting the aforementioned  estimates into \eqref{id1}, utilizing
	 the condition \eqref{a_ij_cond} on the left-hand side,
	  and choosing $\epsilon$ to be sufficiently small, we conclude
	\begin{align*}
		\int_Q e^{-2\hat{c}\gamma_0}\theta^{-2} \left|\nabla(y_\varepsilon-w_\varepsilon)\right|^2 dxdt
		\leq C |z|^2_{C^{1, 0}(\overline{Q})} |y_0|^2_{L^2(\Omega)}.
	\end{align*}
Combining this estimate with (\ref{diff_est_final}) and (\ref{Newadd18}) provides
\begin{eqnarray*}
&&\Big|\tilde{I}_1\Big|_{L^2(Q)}
\leq C \left|W^1_\varepsilon\right|_{L^2(Q)}
+ C\left|\theta^{-1} e^{-\hat{c}\gamma_0} (y_\varepsilon-w_\varepsilon)\right|_{L^2(Q)}
+C \left|\theta^{-1} e^{-\hat{c}\gamma_0} \nabla(y_\varepsilon-w_\varepsilon)\right|_{L^2(Q)}\\[2mm]
&&\leq C |z|_{C^{2, 0}(\overline{Q})} |y_0|_{L^2(\Omega)}.
\end{eqnarray*}
Further, by applying (\ref{old*}), (\ref{Newadd33}) and (\ref{521}), it holds that
\begin{eqnarray*}
&&\Big|\tilde{I}_2\Big|_{L^2(Q)} + \Big|\tilde{I}_3\Big|_{L^2(Q)}\\
&&\leq C |z|_{C^{2, 0}(\overline{Q})} \left|\tilde{g}^2_\varepsilon\right|_{W^{2, 0}_2(Q)}
+C |z|_{C^{2, 0}(\overline{Q})} \left( \left|\theta^{-1} y_\varepsilon\right|_{L^2(Q)}
+ \left|\theta^{-1} e^{-\hat{c}\gamma_0} \nabla y_\varepsilon\right|_{L^2(Q)}\right)
\\
&&\leq C |z|_{C^{2, 0}(\overline{Q})} |y_0|_{C^2(\Omega)}.
\end{eqnarray*}
Consequently,  using the above bounds alongside the Sobolev embedding theorem,
we secure:
\begin{eqnarray}\label{522}
\begin{array}{rl}
	&\displaystyle\left|\tilde g^2_\varepsilon-\tilde h^2_\varepsilon\right|_{L^{r_1}(Q)}
	\leq C \left|\tilde g^2_\varepsilon-\tilde h^2_\varepsilon\right|_{W^{2, 1}_2(Q)}
	\displaystyle \leq C \left( \left|\tilde{I}_1\right|_{L^2(Q)} + \left|\tilde{I}_2\right|_{L^2(Q)} + \left|\tilde{I}_3\right|_{L^2(Q)}\right) \\
	&\displaystyle \leq C |z|_{C^{2, 0}(\overline{Q})} |y_0|_{C^{2}(\overline\Omega)}.
	\end{array}
\end{eqnarray}
Furthermore, applying the $L^p$-estimates for linear parabolic
equations to \eqref{58} and \eqref{59} for $k=2$ grants
$$
\left|U^2_\varepsilon\right|_{W^{2, 1}_{r_1}(Q)} \leq C \left( \left|g^2_\varepsilon\right|_{L^{r_1}(Q)} +
\left|\tilde{g}^2_\varepsilon\right|_{L^{r_1}(Q)} \right)
\leq C |y_0|_{C^2(\overline{\Omega})},
$$
and
$$
\left|V^2_\varepsilon\right|_{W^{2, 1}_{r_1}(Q)} \leq C \left( \left|h^2_\varepsilon\right|_{L^{r_1}(Q)} +
\left|\tilde{h}^2_\varepsilon\right|_{L^{r_1}(Q)} \right)
\leq C |y_0|_{C^2(\overline{\Omega})}.
$$

In order to assess $\left|W^2_\varepsilon\right|_{W^{2, 1}_{r_1}(Q)}$,  we first observe that
\begin{equation}\label{Newadd50}
\left| \sum\limits_{i, j=1}^{n} \left(\widehat{a}^{i j}(z, \nabla z) U^2_{\varepsilon, x_i}\right)_{x_j}
\right|_{L^{r_1}(Q)}
\leq C |z|_{C^{2, 0}(\overline{Q})} \left|U^2_\varepsilon\right|_{W^{2, 1}_{r_1}(Q)}
\leq C |z|_{C^{2, 0}(\overline{Q})} |y_0|_{C^2(\overline{\Omega})},
\end{equation}
and
\begin{eqnarray}\label{518}
\begin{array}{ll}
&\displaystyle
\Big|
\tilde a(z)U^2_\varepsilon-\mbox{div} (\tilde B(z)U^2_\varepsilon)
\Big|_{L^{r_1}(Q)}\\[3mm]
&\displaystyle\leq
C|z|_{C^{1, 0}(\overline{Q})}|U^2_\varepsilon|_{L^{r_1}(Q)}+
C|z|_{C^{1, 0}(\overline{Q})}|\nabla U^2_\varepsilon|_{L^{r_1}(Q)}
+C|z|_{C^{2, 0}(\overline{Q})}|U^2_\varepsilon|_{L^{r_1}(Q)}\\[3mm]
&\displaystyle\leq
C |z|_{C^{2, 0}(\overline{Q})} |U^2_\varepsilon|_{W^{2, 1}_{r_1}(Q)}
\leq C|z|_{C^{2, 0}(\overline{Q})} |y_0|_{C^2(\overline{\Omega})}.
\end{array}
\end{eqnarray}

Applying the $L^p$-estimates for linear parabolic equations  to  (\ref{510}) for $k=2$,
and utilizing  (\ref{517}), (\ref{522}), (\ref{Newadd50}) and (\ref{518}), we establish that   $W^2_\varepsilon\in  W^{2, 1}_{r_1}(Q)$ with
$$
\left|W^2_\varepsilon\right|_{W^{2, 1}_{r_1}(Q)}
\leq C |z|_{C^{2, 0}(\overline{Q})} |y_0|_{C^{2}(\overline\Omega)}.
$$
By the Sobolev  embedding theorem,   for $r_2=r_1(n+2)/(n+2-2 r_1)$ $(\mbox{if } n+2>2r_1)$  or
$r_2>1$ $(\mbox{if } n+2\leq2 r_1)$, we conclude
$$
\left|W_\varepsilon^2\right|_{L^{r_2}(Q)}\leq
C\left|W_\varepsilon^2\right|_{W^{2,  1}_{r_1}(Q)}
\leq C |z|_{C^{2, 0}(\overline{Q})} |y_0|_{C^{2}(\overline\Omega)}.
$$
Similarly, reference  (\ref{ll}), we verify
$$
\left|u_\varepsilon\right|_{L^{r_2}(Q)} + \left|v_\varepsilon\right|_{L^{r_2}(Q)}
\leq
C \left( \left|U_\varepsilon^2\right|_{W^{2, 1}_{r_1}(Q)} + \left|V^2_\varepsilon\right|_{W^{2, 1}_{r_1}(Q)}\right)
\leq C
|y_0|_{C^{2}(\overline\Omega)}.
$$

Iterating  this boost  procedure,   we  can  identify  a sufficiently large  $k\in\mathbb{N}$
such that  $U^k_\varepsilon,  V^k_\varepsilon,  W^k_\varepsilon\in
W^{2, 1}_{r_{k}}(Q)$ with $r_k>\frac{n+2}{1-\delta}$.
Thanks to the Sobolev embedding theorem,  this implies  that
$U^k_\varepsilon,  V^k_\varepsilon,  W^k_\varepsilon\in
  C^{1+\delta,
\frac{1+\delta}{2}}(\overline{Q})$
satisfying
$$
\left|W^k_\varepsilon\right|_{C^{1+\delta, \frac{1+\delta}{2}}(\overline{Q})}\leq
C
 |z|_{C^{2, 0}(\overline{Q})}
|y_0|_{C^{2}(\overline\Omega)},
$$
and
$$
\left|U^k_\varepsilon\right|_{C^{1+\delta, \frac{1+\delta}{2}}(\overline{Q})}
+ \left|V^k_\varepsilon\right|_{C^{1+\delta, \frac{1+\delta}{2}}(\overline{Q})}
\leq C |y_0|_{C^{2}(\overline\Omega)}.
$$
Substituting these two estimates back into the analytical definitions for  $u_\varepsilon$ and $v_\varepsilon$  specified in  \eqref{ll}, we secure
$u_\varepsilon,  v_\varepsilon\in  C^{1+\delta, \frac{1+\delta}{2}}(\overline{Q})$, confirming
 the bounds:
\begin{equation}\label{523}
|u_\varepsilon|_{C^{1+\delta, \frac{1+\delta}{2}}(\overline{Q})}
+|v_\varepsilon|_{C^{1+\delta, \frac{1+\delta}{2}}(\overline{Q})}\leq
C
|y_0|_{C^{2}(\overline\Omega)},
\end{equation}
and
\begin{equation}\label{524}
|u_\varepsilon-v_\varepsilon|_{C^{1+\delta, \frac{1+\delta}{2}}(\overline{Q})}\leq
C
 |z|_{C^{2, 0}(\overline{Q})}
|y_0|_{C^{2}(\overline\Omega)}.
\end{equation}

{\bf Step 5. } Finally,  passing to the limit as $\varepsilon$ tends to zero,
by (\ref{old*}), (\ref{523}) and (\ref{524}), we
obtain  controls $u^*, v^*\in C^{1+\delta, \frac{1+\delta}{2}}(\overline{Q})$
such that the corresponding solutions $y^*$ and $w^*$ to
(\ref{old1}) and (\ref{aeq1*}) satisfy
$y^*(x, T)=w^*(x, T)=0$ in $\Omega$. Furthermore,
the controls $u^*$ and  $v^*$ satisfy
$$|u^*|_{C^{1+\delta, \frac{1+\delta}{2}}(\overline{Q})}+|v^*|_{C^{1+\delta, \frac{1+\delta}{2}}(\overline{Q})}\leq
C
|y_0|_{C^{2}(\overline\Omega)},
$$
and
$$
|u^*-v^*|_{C^{1+\delta, \frac{1+\delta}{2}}(\overline{Q})}\leq
C
 |z|_{C^{2, 0}(\overline{Q})}
|y_0|_{C^{2}(\overline\Omega)}.
$$
This completes the proof of \Cref{t3}.
\end{proof}

Finally, we conclude the proof for \Cref{th2} using the fixed point method and \Cref{t3}.
\begin{proof}[Proof of \Cref{th2}]
First, employing  the arguments analogous to  those utilized
 in  \cite[Theorem  1.1]{Liu} and \cite[Theorem  1.1]{Liux}, we establish a local null controllability result for the quasi-linear  parabolic equation (\ref{aeq1}).
In fact, for any $z\in K$,  we define
 \begin{eqnarray*}
 	&&\mc{U}(z) := \bigg\{\ u\in  C^{1+\delta, \frac{1+\delta}{2}}(\overline{Q}) \  \Big| \  \left|u\right|_{C^{1+\delta, \frac{1+\delta}{2}}(\overline{Q})} \leq C \left|y_0\right|_{C^2(\overline{\Omega})} \ \mbox{and }\\
 	&&\quad\quad\quad\quad\quad\quad\quad\quad\quad\quad\quad\quad\text{the associated solution $y$ to \eqref{old1} satisfies $y(\cdot, T)=0$ in $\Omega$}\  \bigg\},
 \end{eqnarray*}
 and
 \begin{align*}
 	\mc{Y} (z) := \Big\{\ y\in K \ \Big|\  y \text{ solves \eqref{old1} for some } u\in \mc{U}(z)\ \Big\}.
 \end{align*}
 %
%
This defines a  multi-valued mapping $\mc{Y} : K\rightarrow 2^K$, provided
$y_0$ is sufficiently small. Indeed, by the local well-posedness  for (\ref{old1}), we obtain
$$
|y|_{C^{3+\delta, \frac{3+\delta}{2}}(\overline{Q})} \leq C\left(
|u|_{C^{1+\delta, \frac{1+\delta}{2}}(\overline{Q})}+|y_0|_{C^{3+\delta}(\overline\Omega)}\right)
\leq C|y_0|_{C^{3+\delta}(\overline\Omega)}.
$$
Hence, there exists a sufficiently small $\rho_2>0$ such that
 when $|y_0|_{C^{3+\delta}(\overline\Omega)}\leq \rho_2$,
$|y|_{C^{3+\delta, \frac{3+\delta}{2}}(\overline{Q})} \leq 1$, which implies  $\mc{Y}(K)\subseteq K$.

Applying the Kakutani fixed-point theorem (following \cite[Theorem  1.1]{Liu} and
 \cite[Theorem  1.1]{Liux}),
there exists a control $u^*\in C^{1+\delta, \frac{1+\delta}{2}}(\overline{Q})$ satisfying $\left|u^*\right|_{C^{1+\delta, \frac{1+\delta}{2}}(\overline{Q})} \leq C \left|y_0\right|_{C^2(\overline{\Omega})}$ such that
 the associated solution $y^*$ to (\ref{aeq1}) satisfies  $y^*(x, T)=0$ in $\Omega$.
By Theorem \ref{t3}, there also exists a control
 $v^*\in C^{1+\delta, \frac{1+\delta}{2}}(\overline{Q})$ such that
 the associated solution $w^*$ to (\ref{aeq1*}) satisfies  $w^*(x, T)=0$ in $\Omega$.

Furthermore,
 the resultant  solution $y^*\in K$ to (\ref{aeq1}) can
  be regarded as the solution to the linearized system (\ref{old1})
associated to $z=y^*$ and $u=u^*$.  Utilizing  Theorem \ref{t3} and applying
 the Schauder estimates of
linear parabolic equations for (\ref{old1}) (see \cite[Theorem 4.28]{Li}), we derive
\begin{eqnarray}
\begin{array}{rl}
&\displaystyle|u^*-v^*|_{C^{1+\delta, \frac{1+\delta}{2}}(\overline{Q})}\\
&\displaystyle\leq  C
 |y^*|_{C^{2, 0}(\overline{Q})}
|y_0|_{C^{2}(\overline\Omega)}\leq C \left(|u^*|_{C^{\delta, \frac{\delta}{2}}(\overline{Q})}+|y_0|_{C^{2+\delta}(\overline{\Omega})}
\right) |y_0|_{C^{2}(\overline\Omega)}\leq C
|y_0|^{2}_{C^{2+\delta}(\overline\Omega)}.
\end{array}
\end{eqnarray}

Moreover,
following arguments similar to  the proof of Proposition \ref{th1},
the state deviation is quantified as:
\begin{eqnarray*}
&&|y^*-w^*|_{C^{3+\delta, \frac{3+\delta}{2}}(\overline{Q})}\\
&&\leq C\left(|u^*-v^*|_{C^{1+\delta, \frac{1+\delta}{2}}(\overline{Q})}+
\big|\sum_{i, j=1}^{n}(\widehat{a}^{i j}(y^*, \nabla y^*) y^*_{x_i})_{x_j}\big|_{C^{1+\delta, \frac{1+\delta}{2}}(\overline{Q})}+|f(y^*, \nabla y^*)|_{C^{1+\delta, \frac{1+\delta}{2}}(\overline{Q})}\right)\\
&&\leq C|u^*-v^*|_{C^{1+\delta, \frac{1+\delta}{2}}(\overline{Q})}
+C\sum_{i, j=1}^{n}|\widehat{a}^{i j}(y^*, \nabla y^*)|_{C^{1+\delta, \frac{1+\delta}{2}}(\overline{Q})}|y^*_{x_i x_j}|_{C^{1+\delta, \frac{1+\delta}{2}}(\overline{Q})}\\
&&\quad+C\sum_{i, j=1}^{n}|(\widehat{a}^{i j}(y^*, \nabla y^*))_{x_j}|_{C^{1+\delta, \frac{1+\delta}{2}}(\overline{Q})} |y^*_{x_i}|_{C^{1+\delta, \frac{1+\delta}{2}}(\overline{Q})}
+C|f(y^*, \nabla y^*)|_{C^{1+\delta, \frac{1+\delta}{2}}(\overline{Q})}\\
&&\leq C|u^*-v^*|_{C^{1+\delta, \frac{1+\delta}{2}}(\overline{Q})}+C|y^*|^2_{C^{3+\delta, \frac{3+\delta}{2}}(\overline{Q})}\\
&&\leq C\left(|u^*-v^*|_{C^{1+\delta, \frac{1+\delta}{2}}(\overline{Q})}+
|u^*|^2_{C^{1+\delta, \frac{1+\delta}{2}}(\overline{Q})}+
|y_0|^2_{C^{3+\delta}(\overline{\Omega})}\right)\leq
C
|y_0|^{2}_{C^{3+\delta}(\overline{\Omega})}.
\end{eqnarray*}
This completes the proof of Theorem \ref{th2}.

\end{proof}

\begin{remark}
	Note that the null controls obtained in this proof do not have the vanishing property at both ends as stated in \Cref{lemma1!}, because of the difference in the weight function considered for the Carleman estimate in \Cref{chapter2} with the one considered in the article \cite{DLZ}. Consequently, the null controls obtained in this method cannot be used to establish the error estimate \eqref{approx_est} via the time-splitting method described in \Cref{chapter4}.
\end{remark}


\end{document}